\documentclass[12pt,reqno]{amsart}
\usepackage[all,cmtip]{xy}
\usepackage{tikz-cd}
\usepackage{parskip}
\usepackage{amssymb}

\usepackage{hyperref}

\numberwithin{equation}{section}

\newtheorem{theorem}{Theorem}[section]
\newtheorem{corollary}[theorem]{Corollary}
\newtheorem{lemma}[theorem]{Lemma}
\newtheorem{proposition}[theorem]{Proposition}

\theoremstyle{definition}

\newtheorem{example}[theorem]{Example}

\newtheorem{remark}[theorem]{Remark}

\begin{document}
	
	\title[On the Bott-Samelson-Demazure-Hansen varieties]{On the 
		geometry of the anti-canonical bundle of the Bott-Samelson-Demazure-Hansen varieties}
	
	\author[I. Biswas]{Indranil Biswas}
	
	\address{Tata Institute of Fundamental Research,
		Homi Bhabha Road, Colaba
		Mumbai 400005, India.}
	
	\email{indranil@math.tifr.res.in}
	
	\author[S. S. Kannan]{S. Senthamarai Kannan}
	
	\address{Chennai Mathematical Institute, Plot H1, SIPCOT IT Park, Siruseri, Kelambakkam, 603103, India.}
	
	\email{kannan@cmi.ac.in}
	
	\author[P. Saha]{Pinakinath Saha}
	
	\address{Tata Institute of Fundamental Research,
		Homi Bhabha Road, Colaba
		Mumbai 400005, India.}
	
	\email{psaha@math.tifr.res.in}
	
	
	
	\thanks{}
	
	\begin{abstract}
		Let $G$ be a semi-simple simply connected algebraic group over the field $\mathbb{C}$ of 
		complex numbers. Let $T$ be a maximal torus of $G,$ and let $W$ be the Weyl group of $G$ with 
		respect to $T$. Let $Z(w,\, \underline{i})$ be the Bott-Samelson-Demazure-Hansen variety 
		corresponding to a tuple $\underline{i}$ associated to a reduced expression of an element $w \,\in\, W.$ We 
		prove that for the tuple $\underline{i}$ associated to any reduced expression of a 
		minuscule Weyl group element $w,$ the anti-canonical line bundle on $Z(w,\,\underline{i})$ 
		is globally generated. As consequence, we prove
that $Z(w,\,\underline{i})$ is weak Fano.

Assume that $G$ is a simple algebraic group whose type is different from $A_2.$ Let 
$S\,=\,\{\alpha_{1},\,\cdots,\,\alpha_{n}\}$ be the set of simple roots. Let 
$w$ be such that support of $w$ is equal to $S.$ We prove that $Z(w,\,\underline{i})$ is Fano 
for the tuple $\underline{i}$ associated to any reduced expression of $w$ if and only if $w$ is 
a Coxeter element and $w^{-1}(\sum_{t=1}^{n}\alpha_{t})\,\in\, -S$.
\end{abstract}	

\maketitle 
	
\tableofcontents
	
\section{Introduction}

	Let $G$ be a semi-simple simply connected algebraic group over $\mathbb{C}$ and $T$ a maximal torus of $G.$ Let $W\,=\,N_{G}(T)/T$ denote the 
	Weyl group of $G$ with respect to $T.$ We denote the set of roots of $G$ with respect to $T$ by $R$. Let $B^{+}$ be a Borel subgroup of $G$ 
	containing $T$. The Borel subgroup of $G$ opposite to $B^{+}$ determined by $T$ is denoted by
	$B$, in other words, $B\,=\,n_{0}B^{+}n_{0}^{-1}$, where $n_{0}$ 
	is a representative in $N_{G}(T)$ of the longest element $w_{0}$ of $W$. Let $R^{+}\,\subset\, R$ be the set of positive roots of $G$ with 
	respect to the Borel subgroup $B^{+}$. For $\beta \,\in\, R^{+},$ we also use the notation $\beta \,>\, 0$.
	The set of positive roots of $B$ is equal 
	to the set $R^{-} \,:=\, -R^{+}$ of negative roots. Let $S \,=\, \{\alpha_1,\,\cdots,\,\alpha_n\}$ denote the set of simple roots in $R^{+},$
	where $n$ is the rank of $G.$ The simple reflection in $W$ corresponding to $\alpha_i$ is denoted by $s_{i}.$
	
	For any $w\,\in\, W,$ let
	\begin{equation}\label{sv}
		X(w)\,:=\,\overline{BwB}/B
	\end{equation}
	denote the Schubert variety in $G/B$ corresponding to $w.$ Let
	$\underline{i}\,=\,(i_{1},\,\cdots ,\,i_{r})$ be tuple associated with a reduced expression of $w\,=\,s_{i_{1}}\cdots s_{i_{r}}.$
	Let $Z(w ,\,\underline{i})$ be the Bott-Samelson-Demazure-Hansen variety (it is a natural desingularization of $X(w)$) corresponding to
	$(w,\underline{i}).$ It was first introduced by Bott and Samelson in a differential
	geometric and topological context (see \cite{BS}). Demazure in \cite{De} and Hansen in \cite{Han}
	independently adapted the construction in algebro-geometric situation, which explains the name.
	
In \cite{LLM}, Lakshmibai, Littelmann and Magyar studied the Standard Monomial Theory for 
globally generated (respectively, ample) line bundles on $Z(w,\,\underline{i}).$ As a 
consequence, they proved the cohomology vanishing theorems for globally generated (and also 
ample) line bundles. In \cite{LT}, Lauritzen and Thompsen characterized the globally generated, 
ample and very ample line bundles on $Z(w,\,\underline{i}).$ In \cite{CKP}, Chary, Kannan and 
Parameswaran studied the cohomology of tangent bundle on $Z(w,\,\underline{i}),$ they proved 
that the higher cohomology groups of the tangent bundle on 
$Z(w,\,\underline{i})$ vanish for any simply-laced group $G.$

	In \cite{And}, Anderson computed the cone of effective divisors on a Bott-Samelson-Demazure-Hansen variety corresponding to an arbitrary 
	sequence of simple roots. In \cite{Cha}, Chary characterized all the tuples $\underline{i}$ associated to the reduced expressions of $w$ such that 
	$Z(w,\,\underline{i})$ is Fano or weak Fano. In \cite{Cha1}, Chary classified Fano, weak Fano and log Fano Bott-Samelson-Demazure-Hansen varieties and their
toric limits in Kac--Moody setting.

Fix an ordering of the simple roots as in the Dynkin diagram in
\cite[p.~58]{Hum1}. Let $w$ and $\underline{i}$ be as above. Then by the
results in \cite[p.~464, Section 3.1]{LT} we have
	$$
K_{Z(w,\underline{i})}^{-1}\,=\,\mathcal{O}_{1}(m_{1})\otimes \cdots \otimes \mathcal{O}_{r}(m_{r})
$$ for some $(m_{1},\, \cdots ,\, m_{r})\,\in \, \mathbb{Z}^{r}$ (for more precisely see Section \ref{Sec2}).

In view of these works, it is natural to ask the following questions:
	
\begin{enumerate}
\item Is there a formula for $m_{j}$'s in terms of the Cartan integers ?
		
\item Is there a class of $Z(w,\,\underline{i})$ for which the 
anti-canonical line bundle $K_{Z(w,\,\underline{i})}^{-1}\, \longrightarrow\, 
Z(w,\,\underline{i})$ is globally generated?

\item For which $Z(w,\,\underline{i}),$ do all the higher cohomology groups of 
$K_{Z(w,\underline{i})}^{-1}$ vanish?

\item Is there a criterion for which $Z(w,\,\underline{i})$ is Fano for the 
tuple $\underline{i}$ associated to any reduced expression of $w$?

\item Does the $B$-module $H^{0}(Z(w,\,\underline{i}),\, K_{Z(w,\underline{i})}^{-1})$ contain a 
unique $B$-stable line? If so what is the lowest weight of 
$H^0(Z(w,\,\underline{i}),\,K_{Z(w,\underline{i})}^{-1})?$
\end{enumerate}

In this article, we give answers to some of these questions.

We now state our results.

\begin{proposition}\label{prop1}
	For any integer $1\,\le\, j\,\le \,r,$ we have $m_{j}\,=\,\langle \sum_{l=j}^{r}\alpha_{i_{l}},\,\alpha_{i_{j}} \rangle$ if there is no integer
	$j+1\,\le\, k\,\le\, r$ such that $i_{k}\,=\,i_{j}.$ Otherwise, $m_{j}\,=\,\langle \sum_{l=j}^{k-1}\alpha_{i_{l}},\,
	\alpha_{i_{j}}\rangle$ where $k$ is the least integer $j+1\,\le\, k\,\le\, r$ such that $i_{k}\,=\,i_{j}$. $($ $\langle \cdot, \cdot \rangle$ is defined in \eqref{eq2.1} $)$.
\end{proposition}

A fundamental weight $\omega$ is said to be minuscule if $\omega$ satisfies 
$\langle \omega,\, \beta \rangle\,\le \,1$ for all $\beta \,\in\, R^{+}.$

If $\omega_{m}$ is a minuscule fundamental weight corresponding to the simple 
root $\alpha_{m},$ then the standard parabolic subgroup $P_{S\setminus\{\alpha_{m}\}}$ of $G$ 
corresponding to the subset $S\setminus\{\alpha_{m}\}$ of $S$ is called minuscule maximal 
parabolic subgroup of $G.$

Let $W_{S\setminus\{\alpha_{m}\}}$ denote the Weyl group of 
$P_{S\setminus\{\alpha_{m}\}}.$ Let $W^{S\setminus \{\alpha_{m}\}}$ denote the set of minimal 
coset representative of $W/W_{S\setminus\{\alpha_{m}\}}$ in $W.$ For minuscule fundamental 
weight $\omega_{m},$ the elements of $W^{S\setminus \{\alpha_{m}\}}$ are called minuscule Weyl 
group elements.

We prove the following theorem (see Theorem \ref{lem1.3}):
	
	\begin{theorem}\label{Thm1.2}
	 For the tuple $\underline{i}$ associated to any reduced expression  of a minuscule Weyl group element $w,$ the anti-canonical
		line bundle $K_{Z(w,\underline{i})}^{-1}$ on $Z(w,\,\underline{i})$ is globally generated. In particular, all the higher
		cohomologies of $K_{Z(w,\underline{i})}^{-1}$ vanish.
	\end{theorem}

For $w\,\in\, W,$ we define ${\rm supp}(w)\,:=\,\{\alpha_{i}\,\in\, S
\,\big\vert\,\, s_{i}\, {~\rm 
appears~in ~some ~reduced ~expression ~of ~}\, w \}.$

An element $w\,\in\, W$ is said to be a Coxeter element if 
$w\,=\,s_{\sigma(1)}\cdots s_{\sigma(n)}$ for some $\sigma\,\in\, S_{n},$ where $S_{n}$ denotes the 
symmetric group on $n$ letters $1,\,\cdots,\, n.$

We also prove the following theorem (see Theorem \ref{prop1.4}):

\begin{theorem}\label{thm1.3}
Let $G$ be a simple algebraic group whose type is different from $A_2.$ Let $w$ be such that ${\rm supp}(w)\,=\,S.$ Then
$Z(w,\,\underline{i})$ is Fano for the tuple $\underline{i}$ associated to
any reduced expression of $w$ if and only if $w$ is a Coxeter
element and $w^{-1}(\sum_{t=1}^{n}\alpha_{t})\,\in\, -S$.
\end{theorem}

Now a natural question is whether there exists an element $w\,\in\, W$ satisfying
the assertions of Theorem \ref{thm1.3}.
The following gives an affirmative answer for this question (see Lemma \ref{lem 5.9}):
\begin{lemma}
	Let $G$ be a simple algebraic group such that rank of $G$ is at least two. Then the following is the 
(non-constructive) list of Coxeter elements $c$ satisfying 
	$c^{-1}(\sum_{t=1}^{n}\alpha_{t})\,\in\, -S$:
	\begin{itemize}
		\item[(1)] If $G$ is simply-laced, for every simple root $\alpha_{i},$ there is a unique Coxeter element
		$c_{i}$ such that $c_{i}^{-1}(\sum_{t=1}^{n}\alpha_{t})\,=\,-\alpha_{i}.$
		
		\item[(2)] If $G$ is not simply-laced, for every short simple root $\alpha_{i}$, there is a unique Coxeter element $c_{i}$ such
		that $c_{i}^{-1}(\sum_{t=1}^{n}\alpha_{t})\,=\,-\alpha_{i}.$ For long simple root $\alpha_{i}$ there is no such Coxeter element.
	\end{itemize}
\end{lemma}

The following is proved (see Proposition \ref{prop4.0}):

\begin{proposition}\label{prop2}
Let $\mathcal{L}$ be a $B$-linearized line bundle on $Z(w,\,\underline{i})$ such that $H^0(Z(w,\underline{i}),\mathcal{L})\neq 0.$ Then the $B$-module $H^0(Z(w,\underline{i}),\mathcal{L})$ has a unique $B$-stable line $L_{\mu}$ of weight $\mu.$ Moreover, any weight  $\nu$ of this module satisfies $\nu\ge \mu.$
\end{proposition}

The following is proved (see Proposition \ref{prop6.4}).

\begin{proposition}\label{Prop1.5}
Assume that $G\,=\,{\rm SL}(4,\mathbb{C}).$ Let $w_{0}=s_{i_1}s_{i_2}\cdots s_{i_6}$ be a reduced expression and $\underline{i}=(i_1,i_2,\ldots,i_6).$ Then
the anti-canonical line bundle on $Z(w_{0},\,\underline{i})$ is globally generated.
\end{proposition}

\begin{example}
	Let $G$ be of type $A_{n-1}.$ Consider the reduced expression $$w_{0}\,=\,s_{1}(s_{2}s_{1})(s_{3}s_{2}s_{1})\cdots (s_{n-1}s_{n-2}\cdots s_{1}).$$ Let $\underline{i}\,=\,(i_{1},\,\cdots ,\,i_{N})$ be the tuple associated to the above reduced expression of
	$w_{0},$ where $N\,=\,{{n}\choose{2}}.$
	Then the anti-canonical line bundle $K_{Z(w_0,\underline{i})}^{-1}$ on $Z(w_{0},\,\underline{i})$ is globally generated (see Example \ref{prop6.1}). 
\end{example}

Note that the construction of the Bott-Samelson-Demazure-Hansen-variety $Z(w,\,\underline{i})$ depends not only on $w$ but also it 
depends on the choice of the reduced expression $\underline{i}$ of $w.$ Two Bott-Samelson-Demazure-Hansen-varieties corresponding to 
same $w$ but different reduced expressions are not necessarily isomorphic (see \cite{CKP}). We illustrate this fact by showing that 
there are two tuples $\underline{i}$ and $\underline{i}'$ associated to two reduced expressions of $w_{0}$ in $SL(5,\mathbb{C})$ such that the anti-canonical 
line bundle of $Z(w_0,\,\underline{i})$ is globally generated and the anti-canonical line bundle of $Z(w_0,\,\underline{i}')$ is not 
globally generated (see Remark \ref{rm6.7}).
	
All the results of our article work over an algebraically closed field of 
arbitrary characteristics except all the results in Section \ref{subsec3} and Lemma \ref{lem 2.2}, Lemma \ref{lem4.4}, Proposition \ref{prop4.5} and Proposition \ref{prop4.6} in Section \ref{Sec4}.
	
The organization of the article is as follows. In Section 2, we introduce some notation and preliminaries on algebraic groups, Lie 
algebras, and Bott-Samelson-Demazure-Hansen varieties. In Section 3, we recall some results on cohomology of line bundles on Schubert 
varieties. Then we study cohomology of the anti-canonical line bundle on Bott-Samelson-Demazure-Hansen variety and give an inductive 
method to compute all the cohomology groups of the anti-canonical line bundle on Bott-Samelson-Demazure-Hansen variety. In Section 
4, we prove Proposition \ref{prop2}. In Section 5, we prove Proposition \ref{prop1}, Theorem \ref{Thm1.2} and Theorem \ref{thm1.3}. In Section 6, we prove Proposition \ref{Prop1.5}.
	
\section{Preliminaries}\label{Sec2}
	
	In this section some notation and preliminaries are set up. For details on 
	algebraic groups, Lie algebras and Bott-Samelson-Demazure-Hansen varieties \cite{BK05}, \cite{Hum1}, \cite{Hum2}, \cite{Jan} 
	are referred.
	
	Let $\mathfrak{g}$ be the Lie algebra of $G$. 
	Let $\mathfrak{b}\subset \mathfrak{g}$ and $\mathfrak{h}\subset \mathfrak{b}$ be the Lie algebras of $B$ and
	$T$ respectively. The group of all characters of $T$ is denoted by $X(T)$. Therefore,
	$$X(T)\otimes \mathbb{R}\,=\,{\rm Hom}_{\mathbb{R}}(\mathfrak{h}_{\mathbb{R}},\, \mathbb{R})$$
	(the dual of the real form of $\mathfrak{h}$). The positive definite 
	$W$-invariant form on ${\rm Hom}_{\mathbb{R}}(\mathfrak{h}_{\mathbb{R}},\, \mathbb{R})$ 
	induced by the Killing form on $\mathfrak{g}$ is denoted by $(-,\,-)$. 
	For any $\mu\,\in\, X(T)\otimes \mathbb{R}$ and $\alpha\,\in\, R$, denote
	\begin{equation}\label{eq2.1}
	\langle \mu,\, \alpha \rangle \,=\, \frac{2(\mu,\alpha)}{(\alpha,\alpha)}\, .
	\end{equation}
	Take a Chevalley basis $x_{\alpha},\, y_{\alpha}, \,\alpha \,\in\, R,\,$ $h_{\alpha_{i}},\, \alpha_{i}\,\in \,S$ 
	of $\mathfrak{g}.$ For any $\alpha\,\in\, R,$ let $U_{\alpha}$ be the root group associated to $\alpha$;
	for $a\,\in\, \mathbb{C}$, let $u_{\alpha}(a)\, \in\, U_{\alpha}$
	be the corresponding element. Denote by $X(T)^+$ the set of dominant characters of 
	$T$ with respect to $B^{+}$. Let $\rho$ be the half sum of all 
	positive roots of $G$ with respect to $T$ and $B^{+}.$ Note that
$\rho$ is the sum of all fundamental weights of $G$ (see \cite[p. 70, Chapter III, Section 13.3, 
Lemma A]{Hum1}). For any simple root $\alpha_{i}$, we denote the fundamental weight
	corresponding to $\alpha_{i}$ by $\omega_{i}.$

	Let ``$\le$'' be the Bruhat order on $W$(see \cite[p. 118, Section 5.9]{Hum3}). The length of any $w\,\in\, W$ is denoted by $\ell(w)$. 
	
	For a simple root $\alpha \,\in\, S,$ let $n_{\alpha}\,\in\, N_{G}(T)$ be a representative of $s_{\alpha}.$ The
	unique minimal parabolic subgroup of $G$ containing $B$ and $n_{\alpha}$ is denoted by $P_{\alpha}.$
	
	We recall that the Bott-Samelson-Demazure-Hansen variety (BSDH-variety for
short) corresponding to the tuple $\underline{i}\,=\,(i_{1},\,i_{2},\,\cdots,\,i_{r})$ associated to a reduced expression $w\,=\,s_{i_{1}}s_{i_{2}}\cdots s_{i_{r}}$ is defined by $$Z(w,\,\underline{i})\,=\,
	\frac{P_{\alpha_{i_{1}}}\times P_{\alpha_{i_{2}}}\times\cdots \times 
		P_{\alpha_{i_{r}}}}{B\times B\times \cdots \times B},$$
where $\underbrace{B\times B\times \cdots \times B}_{( \rm r-times)}$  acts on $P_{\alpha_{i_{1}}}\times P_{\alpha_{i_{2}}}\times\cdots\times P_{\alpha_{i_{r}}}$
	as follows:
	$$(p_{1},\, p_{2},\, \cdots ,\,p_{r})(b_{1}, \,b_{2},\,\cdots,\, b_{r})\,=\,(p_{1}\cdot b_{1},\, b_{1}^{-1} \cdot p_{2}\cdot b_{2},\,
	\cdots,\, b_{r-1}^{-1}\cdot p_{r}\cdot b_{r})$$
	for all $p_{j}\,\in\, P_{\alpha_{i_{j}}}$, $b_{j}\,\in\, B$ (see \cite[p.~73, Definition 1]{De},
	\cite[Definition 2.2.1, p.64]{BK05}). The equivalence class of $(p_{1},\, \cdots,\,p_{r})$ is denoted by $[p_{1},\,
	\cdots ,\,p_{r}].$
	
	We note that for the tuple $\underline{i}$ associated to each reduced expression of $w,$ the BSDH-variety $Z(w,\,\underline{i})$ is a smooth projective
	variety. The BSDH-varieties are equipped with a $B$-equivariant morphism
	\begin{equation}\label{e1}
\phi_{(w,\underline{i})}\,:\, Z(w,\,\underline{i})\,\longrightarrow\, G/B
\end{equation}
defined by $[p_{1},\,\cdots ,\, p_{r}]\, \longmapsto\, p_{1}\cdots p_{r}B.$ Then
$\phi_{(w,\underline{i})}$ is the natural birational surjective morphism from $Z(w,\, \underline{i})$ to
$X(w)$ (see \eqref{sv}). Moreover, $\phi_{(w,\underline{i})}$ is a rational resolution.
	
	For $\underline{i}'\,=\,(i_{1},\,i_{2},\,\cdots,\, i_{r-1})$, let
	$$f_{r}\,:\,Z(w,\, \underline{i})\,\longrightarrow\, Z(ws_{i_{r}}, \,\underline{i}')$$ be the morphism induced by the projection
	$$P_{\alpha_{i_{1}}}\times P_{\alpha_{i_{2}}}\times\cdots\times P_{\alpha_{i_{r}}}\,\longrightarrow \,
	P_{\alpha_{i_{1}}}\times P_{\alpha_{i_{2}}}\times \cdots \times P_{\alpha_{i_{r-1}}}.$$
	This map $f_{r}$ is a $P_{\alpha_{i_{r}}}/B\,\simeq\, \mathbb{P}^1$-fibration. As a consequence, we have ${\rm Pic}(Z(w,\,\underline{i}))\,\simeq\,\mathbb{Z}^{r}.$
	
	For $1\,\le\, j\,\le \,r,$ define $w_{j}\,:=\,s_{i_{1}}\cdots s_{i_{j}}$ and $\underline{i}_{j}\,:=\,(i_{1},\, \cdots,\, i_{j}).$ 
	Then for every $1\,\le\, j\,\le\, r-1,$ we have the projection map
	$$f_{j}\,:\,Z(w,\,\underline{i})\,\longrightarrow\, Z(w_{j},\,\underline{i}_{j})$$
	defined by $[p_{1},\, \cdots ,\,p_{r}]\,\longmapsto\, [p_{1},\, \cdots ,\, p_{j}].$
	For any $1\,\le\, j\,\le\, r,$ there is a $B$-invariant divisor
	$$X_{j}\,:=\, \{[p_{1},\, \cdots ,\, p_{r}]\,\,\big\vert\,\, p_{j}\,=\,e\}\, \subset\, Z(w,\,\underline{i}),$$
	where $e\, \in\, P_{\alpha_{i_{j}}}$ is the identity element.
	Furthermore, the union $\bigcup_{j=1}^r X_{j}$ is a normal crossing divisor.
	The line bundles ${\mathcal O}_{Z(w,\,\underline{i})} (X_{j})$, $1\,\le\, j\,\le \,r$, form a basis of
	${\rm Pic}(Z(w,\,\underline{i}))$ (see \cite[p.~465, Subsection 3.2]{LT}).
	
	For a $B$-module $V,$ the associated vector bundle on $G/B$ is denoted by $\mathcal{L}(V)$, so 
	$$\mathcal{L}(V)\,=\,G\times_{B}V\,=\,G\times V/\sim$$
	where the action of $B$ is given by $(g,v)\,\sim \,(gb,b^{-1}v)$ for $b\,\in\, B,\,g\,\in\, G$ and $v\,\in\, V.$
	Let $\mathcal{L}(w, V)$ denote the restriction to $X(w)$ of the homogeneous vector bundle $\mathcal{L}(V)
	\,\longrightarrow\, G/B$. The pullback of $\mathcal{L}(w, V)$ to $Z(w,\, \underline{i})$,
	via $\phi_{(w,\underline{i})}$ in \eqref{e1}, is denoted by $\mathcal{O}_{r}(V).$ Since
	this vector bundle $\mathcal{O}_{r}(V)$ on $Z(w,\, \underline{i})$ is the pullback of the homogeneous vector bundle
	$\mathcal{L}(w, V)$, we conclude 
	that the cohomology module
	$$H^j(Z(w,\,\underline{i}),\, \mathcal{O}_{r}(V))\,\simeq\, H^j(X(w), \,\mathcal{L}(w, V))$$
	is independent of the choice of the reduced expression $\underline{i}$, for every $j\,\ge\, 0$ (see \cite[Theorem 3.3.4(b)]{BK05}).
	In view of this, we denote $H^j(Z(w,\, \underline{i}),\, \mathcal{O}_{r}(V))$ by $H^j(w,\, V).$ In particular, if $\lambda$ is
	character of $B,$ then we denote the cohomology groups $H^j(Z(w,\,\underline{i}),\, \mathcal{O}_{r}({\lambda}))$ by
	$H^j(w,\, \lambda).$ For any $1\, \leq\, j\, \leq\, r$, let
	$\mathcal{O}_{j}(\lambda)\,=\,f_{j}^{*}\,\phi^*_{(w_{j},\underline{i}_{j})}\mathcal{L}(w_{j}, \lambda)$ be the pulled back
	line bundle on $Z(w,\,\underline{i})$. Define $\mathcal{O}_{j}(1)\,:=\,\mathcal{O}_{j}(\omega_{i_{j}}).$ Then
	$\{ \mathcal{O}_{j}(1)\}_{ 1\le j\le r}$ is a basis of ${\rm Pic}(Z(w,\,\underline{i})).$
	Take $(m_{1},\,\cdots,\,m_{r})\,\in\, \mathbb{Z}^{r}.$ The line bundle $\mathcal{O}_{1}(m_{1})\otimes \cdots \otimes\mathcal{O}_{r}(m_{r})$
	is very ample (respectively, globally generated) if and only if $m_{j}\,>\,0$ (respectively, $m_{j}\,\ge\, 0$) for all
	$1\,\le\, j\,\le\, r$ (see \cite[p.~464--465, Theorem 3.1, Corollary 3.3]{LT}).

	Now we have two sets of basis for ${\rm Pic}(Z(w,\,\underline{i})).$ These two sets of basis are related by the following transformation
	rule:
	
	\begin{lemma}[{\cite[Section 4.2, Proposition 1]{De}}]\label{lem2.1}
		For a character $\lambda$ of $B,$ and for any integer $1\,\le\, k\,\le\, r,$ there is an isomorphism of line bundles 
		on $Z(w,\,\underline{i})$
		$$\mathcal{O}_{k}(\lambda)\,\simeq \,\mathcal{O}_{Z(w,\underline{i})}(\sum\limits_{l=1}^{k}r_{kl}(\lambda)X_{l})\, ,$$
		where the coefficients are
		$$r_{kl}(\lambda)\,=\,\langle -\lambda,\, s_{i_{k}}\cdots s_{i_{l}}(\alpha_{i_{l}}) \rangle
		\,=\, \langle \lambda,\, s_{i_{k}}\cdots s_{i_{l+1}}(\alpha_{i_{l}}) \rangle.$$
	\end{lemma}

	\begin{corollary}\label{cor2.2}
		Let $w\,=\,s_{i_{1}}\cdots s_{i_{r}}$ be a reduced expression and $\underline{i}\,=\,(i_{1},\,\cdots,\, i_{r}).$ Then
		$$\mathcal{O}_{r}(\rho)\,=\,\mathcal{O}_{Z(w, \underline{i})}
		(\sum\limits_{l=1}^{r}\langle \rho,\, s_{i_{r}}\cdots s_{i_{l+1}}(\alpha_{i_{l}})\rangle X_{l}).$$
	\end{corollary}
	
	\begin{proof}
		This follows immediately from Lemma \ref{lem2.1}.
	\end{proof}

The canonical line bundle $K_{Z(w,\, \underline{i})}$ on $Z(w,\, \underline{i})$ is isomorphic to
$\mathcal{O}_{Z(w,\, \underline{i})}(\sum\limits_{j=1}^{r}-X_{j}) \otimes \mathcal{O}_{r}(-\rho)$ (see \cite[Proof of Proposition 10]{MR85},
\cite[p.~67, Proposition 2.2.2]{BK05}). Consequently, the anti-canonical line bundle $K_{Z(w,\underline{i})}^{-1}$ is isomorphic to
$\mathcal{O}_{Z(w, \underline{i})}(\sum\limits_{j=1}^{r}X_{j}) \otimes \mathcal{O}_{r}(\rho).$ 
	
\begin{corollary}\label{cor2.3}
Let $w\,=\,s_{i_{1}}\cdots s_{i_{r}}$ be a reduced expression and $\underline{i}\,=\,(i_{1},\,\cdots,\,i_{r}).$ Then we have
$K_{Z(w,\underline{i})}^{-1}\,=\,\mathcal{O}_{Z(w, \underline{i})}(\sum\limits_{l=1}^{r}(1+\langle \rho,\,
s_{i_{r}}\cdots s_{i_{l+1}}(\alpha_{i_{l}})\rangle )X_{l}).$ In particular, $K_{Z(w,\underline{i})}^{-1}$ is effective (i.e., $K_{Z(w,
\underline{i})}^{-1}$ is associated to an effective divisor) and
$K_{Z(w,\underline{i})}^{-1}$ is big.
\end{corollary}

\begin{proof}
Let $D\,=\,\sum\limits_{l=1}^{r}(1+\langle \rho,\, s_{i_{r}}\cdots s_{i_{l+1}}(\alpha_{i_{l}})
\rangle )X_{l}.$ Then the line bundle $K_{Z(w,\underline{i})}^{-1}$ on $Z(w,\,\underline{i})$ is
associated to the divisor $D.$

Since $\rho$ is dominant and $Z(w,\,\underline{i})$ is smooth, it follows 
that $D$ is an effective Cartier divisor on $Z(w,\,\underline{i}).$

Note that ${\rm supp}(D)\,=\,\bigcup_{l=1}^{r}X_{l},$ where
${\rm supp}(D)$ denotes the support of $D.$
Since $P_{\alpha_{i_{j}}}\,=\,B\sqcup Bs_{i_{j}}B$ for every $1\,\le\, j\,\le\, r,$ we
have $$Z(w,\,\underline{i})\setminus {\rm supp}(D)\,=\,
\frac{Bs_{i_{1}}B\times Bs_{i_{2}}B\times\cdots\times Bs_{i_{r}}B}{B\times B\times\cdots \times B}.$$
		
Further, $Z(w,\,\underline{i})\setminus {\rm supp}(D)$ is
isomorphic to $\prod_{j=1}^{r}U_{\alpha_{i_{j}}},$ an affine $r$-space (see 
\cite[p.~65, Chapter 2, Section 2.2]{BK05}). Now since $Z(w,\,\underline{i})$ is a smooth projective variety such that $Z(w,
\,\underline{i})\setminus {\rm supp}(D)$ is affine, by \cite[Lemma 5.2]{Cha}, 
$D$ is big. Hence, $K_{Z(w,\underline{i})}^{-1}$ is big.
\end{proof}

	\section{Cohomology of the anti-canonical line bundle on BSDH-variety}\label{Sec3}
	
To describe the results in Section \ref{Sec4}, we recall some results from 
\cite{BKS} and \cite{De}.
	
	\subsection{Cohomology of line bundles on Schubert variety}\label{subsec3}
	
	For any $\lambda\,\in\, X(T),$ the one dimensional $B$-module associated to $\lambda$ will be
	denoted by $\mathbb{C}_{\lambda}$. We now recall a
	result due to Demazure on short exact sequences of $B$-modules:
	
	\begin{lemma}[{\cite[p.~271]{Dem}}]\label{lemma 1.1}
		Let $\alpha$ be a simple root, and let $\lambda\,\in \,X(T)$ be such that $\langle \lambda , \,\alpha \rangle \ge 0.$
		Let $ev\,:\,H^0(s_{\alpha}, \,\lambda)\,\longrightarrow \,\mathbb{C}_{\lambda}$ be the evaluation map. Then the following
		statements hold:
		\begin{enumerate}
			\item[(1)] If $\langle \lambda,\, \alpha \rangle \,=\,0,$ then $H^0(s_{\alpha},\, \lambda)\,\simeq\, \mathbb{C}_{\lambda}.$
			
			\item[(2)] If $\langle \lambda , \,\alpha \rangle \ge 1,$ then $\mathbb{C}_{s_{\alpha}(\lambda)}\,\hookrightarrow\,
			H^0(s_{\alpha},\, \lambda) $, and there is a short exact sequence of $B$-modules:
			$$0\,\longrightarrow\, H^0(s_{\alpha},\, \lambda-\alpha)\,\longrightarrow \,H^0(s_{\alpha},\, \lambda)/\mathbb{C}_{s_{\alpha}(\lambda)}
			\,\longrightarrow \,\mathbb{C}_{\lambda}\,\longrightarrow\, 0.$$ Furthermore, $H^{0}(s_{\alpha},\, \lambda- \alpha)\,=\,0$
			if $\langle\lambda ,\, \alpha \rangle\,=\,1.$
			
\item[(3)] Let $m\,=\,\langle \lambda ,\,\alpha \rangle.$ As a $B$-module, $H^0(s_{\alpha},\, \lambda)$ has a composition series 
$$0\,\subseteq\, V_{m}\,\subseteq\, V_{m-1}\,\subseteq\, \cdots \,\subseteq\, V_{0}\,=\,H^0(s_{\alpha},\,\lambda)$$
such that $V_{i}/V_{i+1}\,\simeq\, \mathbb{C}_{\lambda - i\alpha}$ for $i\,=\,0,\,1,\,\cdots,\,m-1$ and
$V_{m}\,=\,\mathbb{C}_{s_{\alpha}(\lambda)}.$
\end{enumerate}
\end{lemma}
	
We define the dot action of $W$ on $X(T)$ by the rule $w\cdot \lambda\,=\, 
w(\lambda + \rho)-\rho,$ where $w\in W,$ and $\lambda\,\in\, X(T).$ The following lemma is an 
immediate consequence of the exact sequences of Lemma \ref{lemma 1.1}.
	
	\begin{lemma}\label{lemma 1.2}
		Let $w\,\in\, W$ and $\alpha$ be a simple root; set $v\,=\,ws_{\alpha}$.
		If $\ell(w) \,=\,\ell(v)+1$, then the following statements hold:
		\begin{enumerate}
			\item If $\langle \lambda ,\, \alpha \rangle \,\geq\, 0$, then 
			$H^{j}(w ,\, \lambda) \,=\, H^{j}(v,\, H^0({s_\alpha,\, \lambda}) )$ for all $j\,\geq\, 0$.
			
			\item If $\langle \lambda ,\,\alpha \rangle \,\geq\, 0$, then $H^{j}(w , \,\lambda ) \,=\, H^{j+1}(w ,\, s_{\alpha}\cdot \lambda)$
			for all $j\,\geq\, 0$.
			
			\item If $\langle \lambda ,\, \alpha \rangle \,\leq\, -2$, then $H^{j+1}(w ,\, \lambda )
			\,=\, H^{j}(w ,\,s_{\alpha}\cdot \lambda)$ for all $j\,\geq\, 0$. 
			
			\item If $\langle \lambda ,\, \alpha \rangle\,=\, -1$, then $H^{j}( w ,\,\lambda)$ vanish for every $j\,\geq\, 0$.
		\end{enumerate}
	\end{lemma}
	
	Henceforth, we will denote the Levi subgroup of $P_{\alpha}$ ($\alpha \,\in\, S$) containing $T$ by $L_{\alpha}$;
	the subgroup $L_{\alpha}\bigcap B\, \subset\, L_\alpha$ will be denoted by $B_{\alpha}.$
	
	The following lemma will be used in the cohomology computations.
	
	\begin{lemma}\label{lemma1.3}
		Let $V$ be an irreducible $L_{\alpha}$-module. Let $\lambda$ be a character of $B_{\alpha}$. Then
		the following statements hold: 
		\begin{enumerate}
			\item As $L_{\alpha}$-modules, $H^j(L_{\alpha}/B_{\alpha},\, V \otimes \mathbb C_{\lambda})\,\simeq\, V \otimes
			H^j(L_{\alpha}/B_{\alpha},\, \mathbb C_{\lambda})$ for every $j\ge 0.$
			
			\item If $\langle \lambda ,\, \alpha \rangle \,\geq\, 0$, then 
			$H^{0}(L_{\alpha}/B_{\alpha} ,\, V\otimes \mathbb{C}_{\lambda})$ 
			is isomorphic, as an $L_{\alpha}$-module, to $V\otimes H^{0}(L_{\alpha}/B_{\alpha} ,\, \mathbb{C}_{\lambda})$. Furthermore, 
			$$H^{j}(L_{\alpha}/B_{\alpha} ,\, V\otimes \mathbb{C}_{\lambda})\, =\,0$$ for every $j\,\geq\, 1$.
			
			\item If $\langle \lambda ,\, \alpha \rangle\,\leq\, -2$, then 
			$H^{0}(L_{\alpha}/B_{\alpha} ,\, V\otimes \mathbb{C}_{\lambda})\,=\,0$, 
			and $H^{1}(L_{\alpha}/B_{\alpha} ,\, V\otimes \mathbb{C}_{\lambda})$
			is isomorphic to $V\otimes H^{0}(L_{\alpha}/B_{\alpha} , \, \mathbb{C}_{s_{\alpha}\cdot\lambda})$. 
			
			\item If $\langle \lambda ,\, \alpha \rangle\,=\, -1$, then 
			$H^{j}( L_{\alpha}/B_{\alpha} ,\, V\otimes \mathbb{C}_{\lambda}) \,=\,0$ 
			for every $j\,\geq\, 0$.
		\end{enumerate}
	\end{lemma}
	
	\begin{proof} Proof (1).\, By \cite[p.~53, Proposition 4.8, I]{Jan} and \cite[p.~77, Proposition 5.12, I]{Jan}, 
		for all $j\,\geq\, 0$, we have the following isomorphism of $L_{\alpha}$-modules:
		$$H^j(L_{\alpha}/B_{\alpha}, \,V \otimes \mathbb C_{\lambda})\,\simeq\, V \otimes
		H^j(L_{\alpha}/B_{\alpha},\, \mathbb C_{\lambda}).$$ 
		
		Statements (2), (3) and (4) follow from Lemma \ref{lemma 1.2} by taking $w\,=\,s_{\alpha}$ and 
		using the fact that $L_{\alpha}/B_{\alpha} \,\simeq\, P_{\alpha}/B$.
	\end{proof}
	
	We recall a result on the structure of the indecomposable $B_{\alpha}$-modules.
	
	\begin{lemma}[{\cite[p.~130, Corollary 9.1]{BKS}}]\label{lemma 1.4}
		Any finite dimensional indecomposable $B_{\alpha}$-module $V$ is isomorphic to 
		$V^{\prime}\otimes \mathbb{C}_{\lambda}$ for some irreducible representation
		$V^{\prime}$ of $L_{\alpha}$ and some character $\lambda$ of $B_{\alpha}$.
	\end{lemma}
	
	\subsection{Cohomology of the anti-canonical line bundle on BSDH-variety}
	
	In this subsection we establish an inductive method for computing the cohomology groups of the anti-canonical
	line bundle on a BSDH-variety.
	
	Let $w\,=\,s_{i_{1}}s_{i_{2}}\cdots s_{i_{r}}$ be a reduced expression of $w$; set $\underline{i}\,:=\,(i_{1},\, i_{2},\,
	\cdots ,\, i_{r}).$ Let $u\,=\,s_{i_{2}}\cdots s_{i_{r}}$ and $\underline{i}'\,=\,(i_{2},\,\cdots ,\,i_{r})$;
	in particular, $\ell(u)\,=\,\ell(w)-1.$
	
Let $$f_{1}\,:\, Z(w,\, \underline{i})\,\longrightarrow \,P_{\alpha_{i_{1}}}/B$$ be the natural projection; it is a
$Z(u,\,\underline{i}')$-fibration.
	Let $\mathcal{R}\, \longrightarrow\, Z(w, \,\underline{i})$ be the relative tangent bundle for the projection $f_{1}.$
	So the restriction of $\mathcal{R}$ to $Z(u,\,\underline{i}')$ is the tangent bundle
	$T_{(u,\underline{i}')}$ on $Z(u,\,\underline{i}').$ We denote the tangent bundle of $Z(w,\, \underline{i})$ by $T_{(w, \underline{i})}.$
	The differential of $f_{1}$ produces the short exact sequence
	\begin{equation}\label{e2}
		0\,\longrightarrow\, \mathcal{R} \,\longrightarrow\, T_{(w, \underline{i})}\,\longrightarrow \, {f_{1}}^*T_{P_{\alpha_{i_{1}}}/B}
		\,\longrightarrow\, 0,
	\end{equation} 
	where $T_{P_{\alpha_{i_{1}}}/B}$ is the tangent bundle on $P_{\alpha_{i_{1}}}/B.$	
	
	\begin{lemma}\label{lem3.1}
		Let $u,\,w,\,\underline{i}'$ and $\underline{i}$ be as above. Then the following statements hold:
		\begin{itemize}
			\item[(i)] $H^0(Z(w,\,\underline{i}),\, K_{Z(w,\underline{i})}^{-1})\,=\,H^0(s_{i_{1}},\, H^0(Z(u,\,\underline{i}' ),\,
			K_{Z(u,\underline{i}')}^{-1})\otimes \mathbb{C}_{\alpha_{i_{1}}}).$
			
			\item[(ii)] There is a short exact sequence
			$$0\,\longrightarrow \,H^1(s_{i_{1}},\,H^{j-1}(Z(u,\,\underline{i}' ),\, K_{Z(u,\underline{i}')}^{-1})\otimes \mathbb{C}_{\alpha_{i_{1}}})\,
			\longrightarrow\, H^j(Z(w,\,\underline{i}),\, K_{Z(w,\underline{i})}^{-1})\, \longrightarrow$$
			$$\longrightarrow\, H^0(s_{i_{1}},\,H^{j}(Z(u,\,\underline{i}' ),\, K_{Z(u,\underline{i}')}^{-1})\otimes \mathbb{C}_{\alpha_{i_{1}}})
			\,\longrightarrow\,0$$
			for all $j\,\ge\, 0$.
		\end{itemize}
	\end{lemma}
	
	\begin{proof}
		{}From \eqref{e2} we have $K_{Z(w,\underline{i})}^{-1}\,=\,\bigwedge^{r}T_{(w, \underline{i})}\,=\,
		\bigwedge^{r-1} \mathcal{R}\otimes {f_{1}}^*T_{P_{\alpha_{i_{1}}}/B}.$ By the projection formula
		(see \cite[Chapter III, p.~253, Exercise 8.3]{Har}),
$$R^{j}{f_{1}}_{*}{K_{Z(w,\underline{i})}^{-1}}\,=\, R^{j}{f_{1}}_{*}{\wedge^{r-1}}\mathcal{R}\otimes T_{P_{\alpha_{i_{1}}}/B}.$$
Further, since $f_{1}$ is a $P_{\alpha_{i_{1}}}$-equivariant projective
morphism and the base is $P_{\alpha_{i_{1}}}/B\,\simeq\, \mathbb{P}^{1},$ we have
$R^{j}{f_{1}}_{*}{\wedge^{r-1}}\mathcal{R}\,=\,\mathcal{L}_{P_{\alpha_{i_{1}}}/B}(H^j(Z(u,\,\underline{i}')),
\,K_{Z(u,\underline{i}')}^{-1} )$ (see \cite[Chapter-III, p. 288, Corollary 12.9]{Har}).
		
		The $E_2$ term of the Leray spectral sequence for $f_{1}$ and $K_{Z(w,\underline{i})}^{-1}$ on $Z(w,\,\underline{i})$ is given by 
		$$E_{2}^{p,q}\,=\, H^{p}(P_{\alpha_{i_{1}}}/B,\, R^{q}{f_{1}}_{*}(K_{Z(w,\underline{i})}^{-1}))\,\simeq\, H^{p}(P_{\alpha_{i_{1}}}/B,\,
		R^{q}{f_{1}}_{*}(\wedge^{r-1}\mathcal{R})\otimes T_{P_{\alpha_{i_{1}}}/B} ),$$
		where $p,\, q$ are non-negative integers. Now since the base
is $P_{\alpha_{i_{1}}}/B\,\simeq\,\mathbb{P}^{1},$ by \cite[p. 57, Section 4.1, Eq (4)]{Jan} the
lemma follows.
\end{proof}	
	
	\section{Structure of the $B$-module ${\rm H}^0(Z(w,\,\underline{i}),
\, K_{Z(w,\underline{i})}^{-1})$}\label{Sec4}

Let $w\,=\,s_{i_1}s_{i_2}\cdots s_{i_{r}}$ and $\underline{i}
\,=\,(i_{1},\,i_2,\,\cdots,\,i_{r})$ be the tuple associated to the reduced expression of $w.$
	In this section we study the structure of the $B$-module $H^0(Z(w,\,\underline{i}),\, K_{Z(w,\underline{i})}^{-1}).$ We give a necessary 
	and sufficient condition for the anti-canonical line bundle on $Z(c,\,\underline{i})$ to be globally generated, where $c$ is a Coxeter element 
	and $\underline{i}$ is the tuple associated to a reduced expression of $c.$

For a $T$-module $V,$ and $\mu\,\in\, X(T),$ define $$V_{\mu}\,:=\,\{v\in V\,
\big\vert\,\, t\cdot v\,=\,\mu(t)v\, {\rm ~ for ~all~}\, t\,\in\, T\}.$$
		
	For $v\,\in\, W,$ define $R^{+}(v)\,:=\,\{\beta\in R^{+}\,\,\big\vert\,\, v(\beta)<0 \}.$ Let $w\,=\,s_{i_{1}}\cdots s_{i_{r}}$ be a reduced expression 
	and $\underline{i}\,=\,(i_{1},\,\cdots,\, i_{r}).$ Define
	\begin{equation}\label{e3}
		\lambda_{w}\,:=\,-\sum\limits_{\beta \in R^{+}(w^{-1})}^{} \beta.
	\end{equation}
	
	\begin{lemma}\label{lem4.1} The element $\lambda_{w}$ defined in \eqref{e3} satisfies $\lambda_{w}\,=\,w\cdot 0.$ 
	\end{lemma} 
	
	\begin{proof}
		Note that $$w(\rho)-\rho\,=\,\frac{1}{2}(\sum\limits_{\beta\in R^{+}\setminus R^{+}(w)}w(\beta) +
		\sum\limits_{\beta \in R^{+}(w)} w(\beta))- \frac{1}{2}(\sum\limits_{\beta'\in R^{+}\cap w(R^{+})}\beta' +
		\sum\limits_{\beta' \in R^{+}\setminus w(R^{+})} \beta').$$ But
		$\frac{1}{2}(\sum\limits_{\beta'\in R^{+}\cap w(R^{+})}\beta' +\sum\limits_{\beta' \in R^{+}\setminus w(R^{+})} \beta')
		\,=\,\frac{1}{2}(\sum\limits_{\beta\in R^{+}\setminus R^{+}(w)}w(\beta) -\sum\limits_{\beta \in R^{+}(w)} w(\beta)).$ 
		Therefore, we have
		$$w(\rho)-\rho= \sum\limits_{\beta \in R^{+}(w)} w(\beta)\,=\,- \sum\limits_{\beta \in R^{+}(w^{-1})} \beta.$$ Consequently,
		$\lambda_{w}\,=\,w\cdot 0.$ 
	\end{proof}	
	
	\begin{lemma}\label{lem 1.2}
		Let $w\,\in\, W$ and $\alpha\,\in\, S$ be such that $s_{\alpha}w \,<\, w.$
		Let $v\,=\,s_{\alpha}w.$ Then $$\langle \lambda_{v},\, \alpha \rangle \,\ge\, 0.$$
	\end{lemma}
	
	\begin{proof}
		From Lemma \ref{lem4.1} it follows that $\langle\lambda_{v},\, \alpha \rangle\,=\,\langle v(\rho)-\rho,\, \alpha \rangle\,=\,
		\langle \rho,\, v^{-1}(\alpha) 
		\rangle-1.$ Since $v^{-1}(\alpha)\,\in\, R^{+},$ and  $\rho$ is the sum of all fundamental weights, we have $\langle \rho, \,v^{-1}(\alpha) \rangle-1\ge 0.$
		Thus, $\langle \lambda_{v},\, \alpha \rangle \ge 0.$
	\end{proof}
	
	\begin{lemma}\label{lem 2.2}
		Let $w\,\in\, W$ and $\lambda \,\in\, X(T)$ be such that $H^0(w,\,\lambda)\,\neq\, 0.$ Then there is a unique $B$-stable line $L$ in
		$H^0(w,\, \lambda)$, and furthermore the weight of any non-zero vector in $L$ is $w(\lambda).$
	\end{lemma}
	
	\begin{proof}
		Let $v\,\in\, W$ be a minimal element such that $v(\lambda)$ is dominant. We will prove the lemma using induction on $\ell(v).$
		
		If $\ell(v)\,=\,1,$ then $v\,=\,s_{\alpha}$ for some $\alpha\,\in\, S.$
Also, since $H^0(w,\, \lambda)\,\neq\, 0$ and $s_{\alpha}(\lambda)$
is dominant, using \cite[p.~110, Theorem 3.3(i)]{BKS}, we have $R^{+}(w)\cap R^{+}(v)\,=\,\emptyset.$ Thus, we have
$\ell(ws_{\alpha})\,=\,\ell(w)+1.$
		Consider the short exact sequence of $B$-modules
		\begin{equation}\label{bm1}
			0\,\longrightarrow\,\mathbb{C}_{\lambda}\,\longrightarrow\, H^0(s_{\alpha},\, s_{\alpha}(\lambda))\,\longrightarrow\, Q
			\,\longrightarrow \,0.
		\end{equation}
		Using it we have $$H^0(w, \,\lambda)\,\subseteq\, H^0(w,\,H^0(s_{\alpha},\, s_{\alpha}(\lambda))).$$
		Since $\ell(ws_{\alpha})\,=\,\ell(w)+1,$ using Lemma \ref{lemma 1.2}(1) we have
		$$H^0(w,\,H^0(s_{\alpha}, \,s_{\alpha}(\lambda)))\,=\,H^0(ws_{\alpha},\, s_{\alpha}(\lambda)).$$
		As $s_{\alpha}(\lambda)$ is dominant, there is a unique $B$-stable line $L$ in
		$H^0(ws_{\alpha}, \,s_{\alpha}(\lambda)),$ and the weight of any non-zero vector in $L$ is $w(\lambda).$
		Further, as $H^0(w,\, \lambda)$ is a non-zero submodule of $H^0(ws_{\alpha},\, s_{\alpha}(\lambda)),$ and $L$ is a unique $B$-stable
		line of $H^0(ws_{\alpha},\, s_{\alpha}(\lambda)),$
		this $L$ is a unique $B$-stable line in $H^0(w,\, \lambda)$ such that the weight of any non-zero vector in $L$ is $w(\lambda).$
		So, the base case is done.
		
Now assume that $\ell(v)\,>\,1.$ Then there is an element $\alpha\,\in\, S$ such that $\ell(v)\,=\,\ell(vs_{\alpha})+1.$ 
Since $v(\alpha)\,\in\,
R^{-},$ and $v$ is minimal such that $v(\lambda)$ is dominant, we have 
$\langle \lambda, \, \alpha \rangle=\, \langle v(\lambda), \, v(\alpha) \rangle\,<\,0.$ Thus, $\langle s_{\alpha}(\lambda),\,
\alpha\rangle>\,\,0.$ Therefore, using \eqref{bm1} together 
		with the argument as in the above paragraph we conclude that
		$$H^0(w, \,\lambda)\,\subseteq\, H^0(ws_{\alpha},\, s_{\alpha}(\lambda)).$$
		Therefore, by induction hypothesis there is a unique $B$-stable line $$L\, \subset\, H^0(ws_{\alpha},\, s_{\alpha}(\lambda))$$ such that the weight of 
		any non-zero vector in $L$ is $w(\lambda).$ Therefore, $L$ is the unique 
		$B$-stable line bundle in $H^0(w,\,\lambda),$ and the weight of any non-zero vector in $L$ is $w(\lambda).$
	\end{proof}
	
	\begin{lemma}\label{lem4.4}
	 Every weight $\mu$ of $H^0(Z(w,\,\underline{i}),\, K_{Z(w,\underline{i})}^{-1})$
		satisfies the condition $\mu \,\ge\, w\cdot 0.$
	\end{lemma}
	
	\begin{proof}
		Recall that $\ell(w)\,=\,r.$ First consider the composition $Z(w,\,\underline{i})\,\longrightarrow\, G/B$
		of $\iota\,:\, X(w)\,\hookrightarrow\, G/B$ and
		$\phi_{(w,\underline{i})}\,:\, 
		Z(w,\,\underline{i})\,\longrightarrow\, X(w)$; for convenience, this composition map will also be
		denoted by $\phi_{(w,\underline{i})}.$ Let $$d\phi_{(w,\underline{i})}\, :\, 
		T_{(w,\underline{i})}\,\longrightarrow\, \phi_{(w,\underline{i})}^{*}T_{G/B}$$
		be the differential of $\phi_{(w,\underline{i})}$ (see \eqref{e1}); it is an injective morphism of sheaves.
		Consequently, $\bigwedge^{r}T_{(w,\underline{i})}$ is a subsheaf of $\phi_{(w,\underline{i})}^{*}(\bigwedge^{r}T_{G/B}).$
		On the other hand, the homogeneous vector bundle $T_{G/B}
\,=\,\mathcal{L}(\mathfrak{g/b})$ has a filtration of subbundles such that the successive
quotients are line bundle given by the characters which are the positive roots. Therefore, $\bigwedge^{r}T_{G/B}$ has a filtration of subbundles such that the successive quotients are line bundles given by the characters $\lambda_{A}\,=\, \sum_{\beta \in A}\beta,$ where
		$A$ runs over all subsets of $R^{+}$ having exactly $r$ number of elements.
		Therefore, every weight of $H^0(Z(w,\,\underline{i}),\, K_{Z(w,\underline{i})}^{-1})$ is a weight of $H^0(w,\, \lambda_{A})$
		for some subset $A$ of $R^{+}$ such that the cardinality of $A$ is $r$. Now, by Lemma \ref{lem 2.2}, if $H^0(w,\, \lambda_{A})
		\,\neq\, 0,$ then there is a unique $B$-stable line $L$ in $H^0(w,\, \lambda_{A})$ such that the weight of any non-zero vector in $L$
		is $w(\lambda_{A}).$
		
		It can be shown that $w(\lambda_{A})\,\ge \,w\cdot 0$ for any subset $A$ of $R^{+}$ of
		cardinality $r$. Indeed, $$w(\lambda_{A})\,=\,w(\sum_{\beta \in A\cap R^{+}(w)}\beta )+ w(\sum_{\beta \in A\setminus R^{+}(w)}\beta)$$
		$$\ge\, w(\sum_{\beta \in R^{+}(w)}\beta)+ w(\sum_{\beta \in A\setminus R^{+}(w)}\beta)
		\,\ge\, w(\sum_{\beta \in R^{+}(w)}\beta)\,=\,w\cdot0$$ for any subset $A$ of $R^{+}$ of cardinality $r$.
	\end{proof}
	
	\begin{proposition}\label{prop4.5}
	The $B$-module
		$H^0(Z(w,\, \underline{i}), \,K_{Z(w,\underline{i})}^{-1})$ has a $B$-stable line $L_{w}$ such that
		the weight of any non-zero vector $v\,\in\, L_{w}$ is $w\cdot 0.$
	\end{proposition}
	
	\begin{proof} 
		We prove by induction on $\ell(w).$ If $\ell(w)\,=\,1,$ then $w\,=\,s_{i_{1}}.$ So, we have $\lambda_{w}\,=\,-\alpha_{i_{1}},$ and ${H^0(Z(w, 
			\,\underline{i}),\,K_{Z(w,\underline{i})}^{-1})}_{\lambda_{w}}\,\neq\, 0.$
		
		Now assume that $\ell(w)\,>\,1.$ Take $v\,=\,s_{i_{1}}w\,=\,s_{i_{2}}\cdots s_{i_{r}}.$ By the induction hypothesis we have ${H^0(Z(v,\, 
			\underline{i}'),\,K_{Z(v,\underline{i}')}^{-1})}_{\lambda_{v}}\,\neq\, 0,$ where $\underline{i}'\,=\,(i_{2},\,\cdots ,\,i_{r}).$
		Moreover, $\lambda_{v}$ is the lowest weight of $H^0(Z(v,\, \underline{i}'),\,K_{Z(v, \underline{i}')}^{-1}).$
		
		Now by Lemma \ref{lem3.1}, we have $H^0(Z(w,\, \underline{i}),\,K_{Z(w, \underline{i})}^{-1})\,=\, H^0(s_{i_{1}},\,H^0(Z(v, \,\underline{i}'),
		\,K_{Z(v,\underline{i}')}^{-1})\otimes \mathbb{C}_{\alpha_{i_{1}}}).$ Then $\lambda_{v}+\alpha_{i_{1}}$ is the lowest weight of $H^0(Z(v,\, 
		\underline{i}'),\,K_{Z(v, \underline{i}')}^{-1})\otimes \mathbb{C}_{\alpha_{i_{1}}}.$
		Further, by Lemma \ref{lem 1.2}, we have $$\langle \lambda_{v}+\alpha_{i_{1}},\, \alpha_{i_{1}} \rangle \,\ge\, 2.$$ Therefore, from Lemma 
		\ref{lemma1.3}(2) it follows that $$H^0(s_{i_{1}},\, H^0(Z(v,\, \underline{i}'),\,K_{Z(v, \underline{i}')}^{-1})\otimes 
		\mathbb{C}_{\alpha_{i_{1}}})_{s_{i_{1}}(\lambda_{v}+\alpha_{i_{1}})}\,\neq\, 0.$$ 
By Lemma \ref{lem4.1}, we have $\lambda_{v}\,=\,
v\cdot 0.$ Therefore, by using $v\cdot 0\,=\,v(\rho)-\rho$, we have $\,s_{i_{1}}(\lambda_{v}+
\alpha_{i_{1}})\,=\,\lambda_{w}.$
	\end{proof}

The more general statement of the Lemma \ref{lem4.4} and Proposition 
\ref{prop4.5} with simple and elegant proof of Proposition \ref{prop4.0} is suggested by the 
referee.

\begin{proposition}\label{prop4.0}
Let $\mathcal{L}$ be a $B$-linearized line bundle on $Z(w,\,\underline{i})$ such
that $H^0(Z(w,\,\underline{i}),\,\mathcal{L})\,\neq\, 0.$ Then the $B$-module
$H^0(Z(w,\,\underline{i}),\,\mathcal{L})$ has a unique $B$-stable line $L_{\mu}$ satisfying
the condition that the
weight of any non-zero vector $v\,\in\, L_{\mu}$ is $\mu.$ Moreover, every $\nu\,\in\, X(T)$
such that $H^0(Z(w,\,\underline{i}),\,\mathcal{L})_{\nu}\,\neq\, 0$ satisfies $\nu\,\ge\, \mu.$
\end{proposition}

\begin{proof}
Consider the restriction map $${\rm res\,:\,} H^0(Z(w,\,\underline{i}),\,
\mathcal{L})^{U}\,\longrightarrow \,H^0(U\cdot z,\,\mathcal{L})^{U},$$ where $z
\,=\,[\dot s_{i_{1}},\,\dot s_{i_{2}},\,\ldots,\,
\dot s_{i_{r}}]$ and $\dot{s_{i_{j}}}$ denotes a representative of $s_{i_{j}}$ in
$P_{\alpha_{i_{j}}}.$ Since
$U\cdot z$ is open dense in $Z(w,\,\underline{i}),$ it follows that ${\rm res}$\, is injective.
Now, let $\sigma_1,\,\sigma_2\,\in\, H^0(U\cdot z,\,\mathcal{L})^{U}\setminus\{0\}.$ Then we
have $$\sigma_1(uz)\,=\,\sigma_1(z),\, \ \
\sigma_2(uz)\,=\,\sigma_2(z),$$ and so $\sigma_1\,=\,c\sigma_2$ for some non-zero scalar $c.$
Therefore, $H^0(U\cdot z,
\mathcal{L})^{U}$ is one dimensional. Hence, $H^0(Z(w,\,\underline{i}),\,\mathcal{L})^{U}$ is
one dimensional.

Consequently, the proof of the Proposition follows by the observation that the subspace
$H^0(Z(w,\,\underline{i}),\,\mathcal{L})^{U}$ of $U$-invariant is a non-zero $T$-module; its
$T$-eigenvectors are the $B$-eigenvectors in $H^0(Z(w,\,\underline{i}),\,\mathcal{L}).$
Moreover, if $\sigma\in H^{0}(Z(w,\,\underline{i}),
\mathcal{L})_{\nu}\setminus\{0\},$ then the $U$-module generated by $\sigma$ is $B$-stable.
Hence, it contains
the unique $B$-stable line $L_{\mu}.$ Therefore, by 
\cite[p.~165, Chapter X, Proposition 27.2]{Hum2}, it follows that $\nu\,\ge\, \mu.$
\end{proof}

Let $c\,=\,s_{i_{1}}s_{i_{2}}\cdots s_{i_{n}}$ be a reduced expression of a Coxeter element, and let $\underline{i}\,=\,(i_{1},\,i_{2},\,\cdots,\, i_{n})$ be the
	tuple corresponding to this reduced expression of $c.$
	
	\begin{proposition}\label{prop4.6} For all $j\, \geq\, 1$, $$H^{j}(Z(c,\,\underline{i}), \,K_{Z(c,\underline{i})}^{-1})\,=\,0.$$
	\end{proposition}
	
	\begin{proof}
		Let $v_{r}\,=\,s_{i_{r}}\cdots s_{i_{n}},$ and $\underline{i}_{r}\,=\,(i_{r},\, \cdots,\, i_{n})$ for all $1\,\le
		\, r\,\le \,n.$ To prove the proposition, we show by descending induction on $r$ that  $$H^{j}(Z(v_{r},\,\underline{i}_{r}), \,K_{Z(v,\underline{i}_{r})}^{-1})\,=\,0$$
		 for all $j\, \geq\, 1.$

		 So, the induction hypothesis says that $$H^j(Z(v_{r+1}, \,\underline{i}_{r+1}),\, K_{Z(v_{r+1}, \underline{i}_{r+1})}^{-1})\,=\,0$$ for all $j\,\ge\, 1.$
		
By Lemma \ref{lem3.1}(i) and induction on $r,$ the weights
of $H^0(Z(v_{r+1},\,\underline{i}_{r+1}),\,K_{Z(v_{r+1},\underline{i}_{r+1})}^{-1})
\otimes \mathbb{C}_{\alpha_{i_{r}}}$ are 
		of the form $\sum\limits_{k=r}^{n}\alpha_{i_{k}}-\mu,$ where $\mu \,\in\,
\sum\limits_{k=r+1}^{n}{\mathbb{Z}_{\ge 0}} \alpha_{i_{k}}.$ It follows from
the classification of the Dynkin diagram that $\langle \sum\limits_{k=r+1}^{n}\alpha_{i_{k}},\, \alpha_{i_{r}} \rangle\,\ge\, -3.$
Hence it follows that $\langle 
		\sum\limits_{k=r}^{n}\alpha_{i_{k}}-\mu,\, \alpha_{i_{r}} \rangle\,\ge\, -1$ for all $\mu \,\in\, \sum\limits_{k=r+1}^{n}{\mathbb{Z}_{\ge 0}} 
		\alpha_{i_{k}}.$ By \cite[p. 129, Proposition 9.1]{BKS} and Lemma \ref{lemma 1.4}, $H^0(Z(v_{r+1},\,\underline{i}_{r+1}),\,K_{Z(v_{r+1},\underline{i}_{r+1})}^{-1})\otimes \mathbb{C}_{\alpha_{i_{r}}}$ is a direct sum of one dimensional $B_{\alpha_{i_{r}}}$-modules of the form $\mathbb{C}_{\lambda},$ where $\lambda=\sum\limits_{k=r}^{n}\alpha_{i_{k}}-\mu$ and $\mu \,\in\, \sum\limits_{k=r+1}^{n}{\mathbb{Z}_{\ge 0}} \alpha_{i_{k}}.$ 		
		Therefore, by using Lemma \ref{lemma1.3} we have $H^1(s_{i_{r}}, \,
		H^0(Z(v_{r+1},\,\underline{i}_{r+1}),\, K_{Z(v_{r+1},\underline{i}_{r+1})}^{-1})\otimes \mathbb{C}_{\alpha_{i_{r}}})\,=\,0.$ On the other hand, by 
		induction hypothesis we have $H^j(Z(v_{r+1},\,\underline{i}_{r+1}),\, K_{Z(v_{r+1},\underline{i}_{r+1})}^{-1})\,=\,0$ for all $j\,\ge\, 1.$
		
Thus, using Lemma \ref{lem3.1}(ii) we have $H^1(Z(v_{r},\,
\underline{i}_{r}),\, K_{Z(v_{r},\underline{i}_{r})}^{-1})\,=\,0.$ Therefore, by using Lemma 
		\ref{lem3.1}(ii) for $j\,\ge\, 2$ inductively, we have $H^j(Z(v_{r} ,\,
\underline{i}_{r}),\, K_{Z(v_{r},\underline{i}_{r})}^{-1})\,=\,0$ for all $j\,\ge\, 2.$
Consequently, we have $$H^j(Z(c,\,\underline{i}), \,K_{Z(c,\underline{i})}^{-1})\,=\,0$$ for all $j\,\ge\, 1.$
	\end{proof}

\begin{remark}

\begin{itemize}
	\item [(i)] We note that $Z(c,\,\underline{i})$ is a smooth projective
toric variety.
	
\item [(ii)] The Proposition \ref{prop4.6} follows from \cite[p.~301, Theorem 
6]{LLM} if $K_{Z(c,\underline{i})}^{-1}$ is globally generated. But, the global generation of 
$K_{Z(c,\underline{i})}^{-1}$ fails in the view of the Example \ref{exam5.5}.
\end{itemize}	
\end{remark}

\section{Main Results}

In this section we prove our main results.
	
Let $w\,=\,s_{i_{1}}s_{i_{2}}\cdots s_{i_{r}}$ be a reduced expression and $\underline{i}\,=\,(i_{1},i_{2},\, \cdots ,\,i_{r}).$
Recall that $w_{j}\,=\,s_{i_{1}}s_{i_{2}}\cdots s_{i_{j}}$ and $\underline{i}_{j}\,=\,(i_{1},\,i_{2},\,\cdots,\,i_{j}).$
For each $k$ such that $s_{k}\,\le\, w_{j},$ let $j(k)$ be the largest integer less than or equal to $j$ such that
$i_{j(k)}\,=\,k.$ Let $\{i_{1},\,i_{2},\,\cdots,i_{j}\}\,=\,\{a_{1},\,a_{2},\,\cdots, \,a_{l}\}$ as sets and
$j(a_{1})\,<\,j(a_{2})\,<\,\cdots \,<\,j(a_{l}).$ Then we prove the following. 
\begin{lemma}\label{lem5.1}

Assume that $\lambda\,\in\, X(B).$ Then the line bundle $\mathcal{O}_{j}(\lambda)$ on $Z(w_{j},\,\underline{i}_{j})$ is
isomorphic to
$$\mathcal{O}_{j(a_{1})}(\langle \lambda,\, \alpha_{a_{1}}\rangle )\otimes \cdots\otimes\mathcal{O}_{j(a_{l})}(\langle \lambda,
\,\alpha_{a_{l}}\rangle).$$
\end{lemma}

\begin{proof}
Without loss of generality we assume that $S\,=\,\{\alpha_{a_{1}},\,\cdots,\, \alpha_{a_{l}}\},$ i.e., $l\,=\,n,$ where $n$ is the
rank of $G.$ Recall that there is a natural embedding $${\rm em}\,:\,G/B
\,\hookrightarrow \,\prod_{k=1}^{l}G/P_{S\setminus\{\alpha_{a_{k}}\}}$$ given by $$gB\,\longmapsto\,
 (gP_{S\setminus\{\alpha_{a_{1}}\}},\,\cdots,\, gP_{S\setminus\{\alpha_{a_{l}}\}}).$$ This induces an isomorphism of the Picard
groups $${\rm em}^{*}\,:\,{\rm Pic}(\prod_{k=1}^{l}G/P_{S\setminus\{\alpha_{a_{k}}\}})\,\,\longrightarrow\,\, {\rm Pic}(G/B)$$
defined by $$\boxtimes_{k=1}^{l} \mathcal{L}^{\otimes m_{k}}(\omega_{a_{k}})\,\longmapsto\, \mathcal{L}(\sum_{k=1}^{l}m_{k}\omega_{a_{k}}),$$
where $m_{k}$'s are integers. The inverse map of ${\rm em}^{*}$ is given by $$\mathcal{L}(\lambda)
\,\longmapsto \,\boxtimes_{k=1}^{l} \mathcal{L}^{\otimes \langle \lambda,\,\alpha_{a_k}\rangle }(\omega_{a_{k}}).$$ 

Consider the commutative diagram

$\centerline{
	\xymatrixcolsep{5pc}\xymatrix{Z(w_{j},\underline{i}_{j})\ar[r]^{\phi_{(w_{j},\underline{i}_{j}})}\ar[d]^{f_{j(a_{k})}}&
G/B \ar@{^{(}->}[r]^{{\rm em}}\ar@{=}[d]&\prod_{k=1}^{l} G/P_{S\setminus \{\alpha_{a_{k}}\}}\ar[d]_{p_{k}}\\
Z(w_{j(a_{k})},\,\underline{i}_{j(a_k)})\ar[r]^{\phi_{(w_{j(a_{k})},\underline{i}_{j(a_{k})})}}& G/B\ar[r]^{\pi_{k}}&G/P_{S\setminus\{\alpha_{a_{k}}\}}}}.$
where $p_{k}\,:\, \prod_{k=1}^{l} G/P_{S\setminus \{\alpha_{a_{k}}\}}\,\longrightarrow\,
G/P_{S\setminus\{\alpha_{a_{k}}\}}$ denotes the projection onto the $k$-th factor and $\pi_{k}\,:\,G/B\,
\longrightarrow\, G/P_{S\setminus\{\alpha_{a_{k}}\}}$ is the natural projection map.

Therefore, we have
$$\mathcal{O}_{j}(\lambda)\,:=\,\phi_{(w_{j},\,\underline{i}_{j})}^{*}\mathcal{L}(\lambda)
\,=\,{({\rm em}\circ \phi_{(w_{j},\underline{i}_{j})})}^{*}\boxtimes_{k=1}^{l}
\mathcal{L}(\langle \lambda,\,\alpha_{a_k}\rangle\omega_{a_{k}})
$$
$$
=\,\bigotimes_{k=1}^{l}{(p_{k}\circ {\rm em}\circ \phi_{(w_{j},\underline{i}_{j}})}^{*}\mathcal{L}(\langle \lambda,\,
\alpha_{a_k}\rangle\omega_{a_{k}}).$$ Now the proof follows from the above commutative diagram and the observation
that $${(\pi_{k}\circ \phi_{(w_{j(a_{k})},\,
\underline{i}_{j(a_{k})})}\circ f_{j(a_{k})})}^{*}\mathcal{L}(\langle \lambda,\,
\alpha_{a_k}\rangle\omega_{a_{k}})\,=\,\mathcal{O}_{j(a_{k})}(\langle \lambda,\,
\alpha_{a_{k}}\rangle).$$
\end{proof}

\begin{lemma}\label{lem5.2}
Let $w$ and $\underline{i}$ be as above.
Then the anti-canonical bundle $K_{Z(w,\underline{i})}^{-1}$ of $Z(w,\,\underline{i})$ is isomorphic to
	$\mathcal{O}_{1}(\alpha_{i_{1}})\otimes \cdots \otimes\mathcal{O}_{r-1}(\alpha_{i_{r-1}})\otimes \mathcal{O}_{r}(\alpha_{i_{r}}).$
\end{lemma}
\begin{proof}
Consider the fiber product diagram

$\centerline{\xymatrixcolsep{5pc}\xymatrix{Z(w,\underline{i})\ar[r]^{\phi_{(w,\underline{i})}}\ar[d]^{f_{r-1}}& G/B \ar[d]^{\pi_{\alpha_{i_r}}}\\ Z(w_{r-1},\underline{i}_{r-1})\ar[r]^{\pi_{\alpha_{i_{r}}}\circ\phi_{(w_{r-1},\underline{i}_{r-1})}}& G/P_{\alpha_{i_{r}}}}}$
where $\pi_{\alpha_{i_{r}}}\,:\,G/B\,\longrightarrow\, G/P_{\alpha_{i_{r}}}$ is the natural projection map. Note that $\mathcal{L}(\alpha_{i_{r}})$ is
the relative tangent bundle of $\pi_{\alpha_{i_r}}$. 

Therefore, we have the following short exact sequence of tangent bundles on $Z(w,\,\underline{i}),$ 
	\begin{equation}\label{eq4.3}
0\,\longrightarrow\, \phi_{(w,\underline{i})}^{*}\mathcal{L}(\alpha_{i_{r}})\,\longrightarrow\, T_{(w,\underline{i})}
\,\longrightarrow\, f_{r-1}^{*} T_{(w_{r-1},\underline{i}_{r-1})}\,\longrightarrow\, 0,
	\end{equation}
	
	where $T_{(w,\underline{i})}$ (respectively, $T_{(w_{r-1},\underline{i}_{r-1})}$) denotes the tangent bundle of $Z(w,\underline{i})$ (respectively, of $Z(w_{r-1},\underline{i}_{r-1})$). Hence, by \eqref{eq4.3} we have $K_{Z(w,\underline{i})}^{-1}=\mathcal{O}_{r}(\alpha_{i_{r}})\otimes f_{r-1}^{*}K_{Z_{(w_{r-1},\underline{i}_{r-1})}}^{-1},$ as $\mathcal{O}_{r}(\alpha_{i_{r}})=\phi_{(w,\underline{i})}^{*}\mathcal{L}(\alpha_{i_{r}}).$
	
	By induction we have $K_{Z_{(w_{r-1},\underline{i}_{r-1})}}^{-1}=\mathcal{O}_{r-1}(\alpha_{i_{r-1}})\otimes \mathcal{O}_{r-2}(\alpha_{i_{r-2}})\otimes\cdots\otimes \mathcal{O}_{1}(\alpha_{i_{1}}).$ Therefore, we have $K_{Z(w,\underline{i})}^{-1}=\mathcal{O}_{r}(\alpha_{i_{r}})\otimes \mathcal{O}_{r-1}(\alpha_{i_{r-1}})\otimes \mathcal{O}_{r-2}(\alpha_{i_{r-2}})\otimes\cdots\otimes \mathcal{O}_{1}(\alpha_{i_{1}}),$ where for $1\,\le\, j
\,\le\, r-1,$ the same notation $\mathcal{O}_{j}(\alpha_{i_{j}})$ is used instead
of $(f_{j}\circ f_{j+1}\circ \cdots \circ f_{r})^{*}\mathcal{O}_{j}(\alpha_{i_{j}}).$
\end{proof}

\begin{proposition}\label{lem5.3}
For any integer $1\,\le\, j\,\le\, r,$ we have $m_{j}\,=\,\langle \sum_{k=j}^{r}\alpha_{i_{k}},\,\alpha_{i_{j}} \rangle,$ if there is no integer
$j+1\,\le\, q\,\le\, r$ such that $i_{q}\,=\,i_{j}.$ Otherwise, $m_{j}\,=\,\langle \sum_{k=j}^{q-1}\alpha_{i_{k}},\,
\alpha_{i_{j}}\rangle,$ where $q$ is the least integer $j+1\,\le\, q\,\le\, r$ such that $i_{q}\,=\,i_{j}.$
\end{proposition}

\begin{proof}
	
{\it Observation:}
By using Lemma \ref{lem5.1}, for every integer $1\le k\le r,$ we have $\mathcal{O}_{k}(\alpha_{i_{k}})=\mathcal{O}_{1}(a_{1,k})\otimes \mathcal{O}_{2}(a_{2,k})\otimes \cdots \otimes \mathcal{O}_{k}(a_{k,k}),$ where $a_{t,k}=\langle \alpha_{i_{k}},\alpha_{i_{t}}\rangle$ for any integer $1\le t\le k$ for which there is no integer $t< l\le k$ such that $i_{l}=i_{t}$ and $a_{t,k}=0$ otherwise.

Now, let $1\,\le\, j\,\le\, r$ be any integer. Let $m=i_{j}.$

Case 1: There is no integer $j+1\,\le\, q\,\le\, r$ such that $i_{q}\,=\,i_{j}.$ Then $j$ is the 
largest integer less than or equal to $r$ such that $i_{j}\,=\,m.$

 Therefore, by the above {\it Observation} and Lemma \ref{lem5.2}, we have $m_j=\sum_{k=1}^{r}a_{j,k}=\sum_{k=j}^{r}\langle\alpha_{i_{k}},\alpha_{i_{j}} \rangle.$

Case 2: There is an integer $j+1\,\le\, q\,\le\, r$ such that $i_{q}=i_{j}.$ Let $q$ be
the least such integer. Then 
$j$ is the largest integer less than or equal to $q-1$ such that $i_{j}\,=\,m.$ Hence, by the above {\it Observation} and Lemma \ref{lem5.2} we have $m_{j}\,=\sum_{k=1}^{r}a_{j,k}=\,\langle \sum_{k=j}^{q-1}\alpha_{i_{k}},\,\alpha_{i_{j}}\rangle.$
\end{proof}

We fix the ordering of the simple roots as done in \cite[p.~58]{Hum1}. Let $c$
be a Coxeter element. So, $c\,=\,s_{\sigma(1)}s_{\sigma(2)}\cdots s_{\sigma(n)}$ for some
$\sigma\,\in\, S_{n}.$ Then we have the following corollary.

\begin{corollary}\label{prop4.7}
$K_{Z(c,\underline{i})}^{-1}$ is globally generated if and only if $\langle \sum\limits_{k=r}^{n}
\alpha_{\sigma(k)},\,\alpha_{\sigma(r)}\rangle \,\ge\, 0$ for all $1\,\le\, r\,\le \,n.$
\end{corollary}

\begin{proof}
Follows from Proposition \ref{lem5.3} and \cite[p.~465, Corollary 3.3]{LT}.
\end{proof}

\begin{example}\label{exam5.5}
	Assume that $G\,=\,{\rm Spin}(8,\mathbb{C}).$ Take $c\,=\,s_{2}s_{1}s_{3}s_{4}$ and $\underline{i}\,=\,(2,\,1,\,3,\,4).$
	Then $K_{Z(c,\underline{i})}^{-1}$ is not globally generated.
\end{example}

\begin{theorem}\label{prop1.4}
	Let $G$ be a simple algebraic group whose type is different from 
$A_2.$ Let $w\,\in\, W$ be such that ${\rm supp}(w)\,=\,S.$ 
Then $Z(w,\,\underline{i})$ is Fano for the tuple $\underline{i}$ associated to any
reduced expression of $w$ if and only if $w$ is a Coxeter element and 
$w^{-1}(\sum_{t=1}^{n}\alpha_{t})\,\in\, -S.$
\end{theorem}
\begin{proof}
	If $G$ is of type $A_1,$ the statement follows trivially.
	
Layout of the proof: If $w$ is not a Coxeter element we show that there is a 
reduced expression $w\,=\,s_{i_{1}}s_{i_2}\cdots s_{i_r}$ and an integer $1\,\le\, k\,<\,r$ such that 
$\alpha_{i_{k}}\,=\,\alpha_{i_{k+2}}$ and the Dynkin diagram of 
$\{\alpha_{i_{k}},\alpha_{i_{k+1}}\}$ is of type $A_2.$ Further, since $G$ is not of type 
$A_{2}$ we will find a tuple $\underline{i}'$ associated to a reduced expression of $w$ such 
that $Z(w,\,\underline{i}')$ is not Fano.
	
Assume that $G$ is not of type $A_{1}$ and $A_2$ and $Z(w,\,\underline{i})$ is Fano for tuple $\underline{i}$ associated to any reduced 
expression of $w.$ Assume on the contrary that $w$ is not a Coxeter element. Then there is a reduced expression 
$\underline{i}\,=\,(i_{1},\,\ldots,\, i_{r})$ of $w$ and integers $1\,\le\, k\,<\,l\le r$ such that $i_{k}\,=\,i_{l}.$ Let $i_{k}\,=\,e.$ Without loss of 
generality we may assume that there is no integer $k\,<\,j\,<\,l$ such that $i_{j}\,=\,e.$
	
Since $Z(w,\underline{i})$ is Fano, by Proposition \ref{lem5.3}, there is a unique integer $k\,<\,j\,<\,l$ such that $\langle 
\alpha_{i_{j}},\,\alpha_{e}\rangle\,\le\, -1.$ Further, for such $j,$ we have $\langle \alpha_{i_{j}},\,\alpha_{e} \rangle\,=\,-1.$ Up to commuting 
relations, we may assume that $j\,=\,k+1$ and $l\,=\,k+2.$ For simplicity of notation we denote the modified expression using commuting 
relations also by $\underline{i}.$ Since $Z(w,\,\underline{i})$ is Fano, by Proposition \ref{lem5.3}, we have 
$m_{j}\,\ge\, 1.$ Hence, we have $$\langle \alpha_{e},\, \alpha_{i_{j}}\rangle\,=\,\langle \alpha_{i_{l}},\,\alpha_{i_{j}} 
\rangle\,=\,-1.$$ Since the Dynkin diagram is connected and different from $A_{2},$ either there is an integer $m<k$ such that $\langle 
\alpha_{e},\,\alpha_{i_{m}} \rangle\,\le\, -1$ or $\langle\alpha_{i_{j}},\,\alpha_{i_{m}} \rangle\,\le\, -1$ or there exists an integer $m\,>\,l$ such 
that $\langle \alpha_{i_{m}},\,\alpha_{i_{j}}\rangle\,\le\, -1$ or $\langle \alpha_{i_{m}}, \,\alpha_{i_{l}}\rangle\,\le\, -1.$
	
{ Case 1:} There is an integer $m<k$ such that $\langle \alpha_{e},\alpha_{i_{m}}\rangle\le -1.$ Let $m$ be the largest such integer. 
Then, $$m_{m}\,\le\, \langle \alpha_{i_{m}},\,\alpha_{i_{m}}\rangle+\langle \alpha_{i_{k}}, \,\alpha_{i_{m}}\rangle+ \langle 
\alpha_{i_{l}},\,\alpha_{i_{m}}\rangle\,=\,2-2\,=\,0.$$ This is a contradiction to $Z(w,\,\underline{i})$ being Fano.
	
	{ Case 2:} There is an integer $m>l$ such that $\langle \alpha_{i_{m}}, \alpha_{i_{j}}\rangle\le -1.$ Let $m$ be the least such integer. Therefore, $m_{j}\,\le\, \langle \alpha_{i_{j}},\,\alpha_{i_{j}}\rangle+\langle \alpha_{i_{l}},\,\alpha_{i_{j}}\rangle+
\langle \alpha_{i_{m}},\,\alpha_{i_{j}}\rangle\,\le\, 0.$ This is a contradiction to $Z(w,\,\underline{i})$ being Fano.
	
	{Case 3:} There exists an integer $m<k$ such that $\langle \alpha_{i_{j}},\alpha_{i_{m}}\rangle\le -1.$ Let $m$ be the largest such integer. Since $\langle \alpha_{i_{j}},\alpha_{e} \rangle=\langle \alpha_{e},\alpha_{i_{j}}\rangle=-1,$ we have $s_{i_{k}}s_{i_{k+1}}s_{i_{k}}=s_{i_{k+1}}s_{i_{k}}s_{i_{k+1}}.$ Therefore, we take the reduced expression $\underline{i}'=(i_{1}',\ldots, i_{r}')$ with $i_{t}'=i_{t}$ for all $t\neq k,k+1,k+2$ and $i_{k}'=i_{k+1},$ $i_{k+1}'=i_{k},$ and $i_{k+2}'=i_{k+1}.$
	
	Hence, for the reduced expression $\underline{i}'$ of $w,$ we have $$m_{m}\,\le\, \langle \alpha_{i_{m}},\,
\alpha_{i_{m}}\rangle+\langle \alpha_{i_{j}},\,\alpha_{i_{m}} \rangle+\langle \alpha_{i_{j}},\,\alpha_{i_{m}} \rangle\,\le\, 0.$$ This is a contradiction to $Z(w,\underline{i}')$ being Fano.

	{ Case 4:} There exists an integer $m>l$ such that $\langle \alpha_{i_{m}}, \,\alpha_{i_{l}}\rangle\,\le\, -1.$ Let $\underline{i}'=(i_{1}',\ldots, i_{r}')$ be the reduced expression as in case 3. Hence, there is an integer $m>l$ such that $\langle \alpha_{i_{m}'},\alpha_{i_{l}}\rangle=\langle \alpha_{i_{m}},\alpha_{e} \rangle\le -1.$ Let $m$ be the least such integer. 
	
	Now, for the reduced expression $\underline{i}'$ of $w$ we have $m_{k+1}=\langle \alpha_{i_{k+1}'}, \alpha_{i_{k+1}'}\rangle+\langle \alpha_{i_{k+2}'}, \alpha_{i_{k+1}'}\rangle+\langle \alpha_{i_{m}'}, \alpha_{i_{k+1}'}\rangle.$ Now note that $\langle \alpha_{i_{k+1}'}, \alpha_{i_{k+1}'}\rangle+\langle \alpha_{i_{k+2}'}, \alpha_{i_{k+1}'}\rangle+\langle \alpha_{i_{m}'}, \alpha_{i_{k+1}'}\rangle=\langle \alpha_{e}, \alpha_{e}\rangle+\langle \alpha_{i_{j}}, \alpha_{e}\rangle+\langle \alpha_{i_{m}}, \alpha_{e}\rangle\le 0.$ This is a contradiction to $Z(w,\underline{i}')$ being Fano.
	
	Thus, $w$ is a Coxeter element.
	
	So, we have $w=s_{i_{1}}\cdots s_{i_{n}}$ where $n={\rm rank}(G)$ and $i_{j}\neq i_{k}$ whenever $j\neq k$. Since 
	$Z(w,\,\underline{i})$ is Fano, by Proposition \ref{lem5.3} for every integer $1\le j\le r-1,$ there is at most one integer $j+1\le k\le r$ such that $\langle \alpha_{i_{k}},\alpha_{i_{j}}\rangle\neq 0.$ Further, for every such integer $k,$ we have $\langle \alpha_{i_{k}},\alpha_{i_{j}}\rangle=-1.$ Hence, the vertex in the Dynkin diagram corresponding to $\alpha_{i_{1}}$ is an end vertex. Therefore, $s_{i_{1}}(\sum_{t=1}^{n}\alpha_{t})=\sum_{t\neq i_{1}}^{}\alpha_{t}.$ Now, let $v=s_{i_{1}}w.$ Let $\underline{i}'=(i_{2},\ldots, i_{n}).$ Then $Z(v,\underline{i}')$ is Fano. Further, \text{supp}($v$)=$S\setminus\{\alpha_{i_{1}}\}$ and since $\alpha_{i_{1}}$ is the end vertex, the Dynkin diagram corresponding to $S\setminus \{\alpha_{i_{1}}\}$ is connected. By induction on rank($G$), $v^{-1}(\sum_{t\neq i_{1}}\alpha_{t})\in -(S\setminus\{\alpha_{i_{1}}\}).$ Therefore, $w^{-1}(\sum_{t=1}^{n}\alpha_{t})\in -S.$
	
	Conversely, assume that $w\in W$ is a Coxeter element such that $w^{-1}(\sum_{t=1}^{n}\alpha_{t})\in -S.$ Let $w=s_{i_{1}}\cdots s_{i_{n}}$ be a reduced expression. 
	
	If $G$ is of type $A_{1},$ then $w\,=\,s_{1}$. So, we are done.
	
	Thus, assume that $G$ is of type different from $A_{1}.$ Since $w^{-1}(\sum_{t=1}^{n}\alpha_{t})\,\in\, -S,$ we
have $\langle \sum_{t=1}^{n}\alpha_{t},\,\alpha_{i_{1}} \rangle \,=\,1.$ Therefore, the Dynkin diagram of
$S\setminus\{\alpha_{i_{1}}\}$ is connected. Also, there is a unique integer $2\,\le\, t_{0}\,\le\, r$ such that
$\langle \alpha_{i_{t_0}},\,\alpha_{i_{1}}\rangle\,=\,-1$ and $\langle \alpha_{i_{t}},\,\alpha_{i_{1}}\rangle\,=\,0$ for
all $t\,\neq\, t_{0},\,i_{1}.$ Consequently, $m_{1}\,=\,1.$
	
Let $v\,=\,s_{i_{1}}w$ and $\underline{i}'\,=\,(i_{2},\,\cdots,\, i_{n}).$ Then $v^{-1}(\sum_{t\neq i_{1}}^{}\alpha_{t})
\,\in\, -(S\setminus \{\alpha_{i_{1}}\}).$ So, by induction on the rank($G$), the variety $Z(v,\,\underline{i}')$ is Fano. Note that there is a unique tuple $(m_{1}',\,m_{2}',\,
\cdots ,\,m_{n-1}')\,\in\, \mathbb{Z}^{n-1}$ such that $K_{Z(v,\underline{i}')}^{-1}\,=\,
\mathcal{O}_{1}(m_{1}')\otimes \cdots \otimes \mathcal{O}_{n-1}(m_{n-1}').$ Thus, using
\cite[p.~464, Theorem 3.1]{LT} it follows that $m_{j}'\,\ge\, 1$ for all $1\,\le\, j\,\le\, n-1.$ By
Proposition \ref{lem5.3}, we have $m_{j+1}\,=\,m_{j}'$ for all $2\,\le \,j\,\le\, n.$
Therefore, by \cite[p. 464, Theorem 3.1]{LT}, we conclude that $Z(w,\,\underline{i})$ is Fano.
\end{proof}

\begin{lemma}\label{lem5.10}
Let $G$ be a simple algebraic group such that rank of $G$ is at least two and $c$ be a Coxeter element. Let $c\,=\,s_{i_{1}}\cdots s_{i_{n}}$ be a reduced expression. For
$1\,\le\, r\,\le\, n-1,$ let
$$J_{r}\,:=\,\{\alpha\,\in\, S\setminus \{\alpha_{i_{1}},\,\cdots,\,\alpha_{i_{r-1}}\}\, \,\big\vert\,\,
\langle \sum_{j=r}^{n}\alpha_{i_{j}},\,\alpha\rangle\,=\,1\}.$$ Then $\alpha_{i_{r}}\,\in\, J_{r}$ for all $1\,\le\, r
\,\le\, n-1$ if and only if $c^{-1}(\sum_{t=1}^{n}\alpha_{t})\,\in\, -S.$ 
\end{lemma}

\begin{proof} We argue by induction on the rank of $G.$ Note
that $\alpha_{i_{1}}\,\in\, J_{1}$ if and only if
$\langle \sum_{j=1}^{n}\alpha_{i_{j}},\,\alpha_{i_{1}}\rangle\,=\,1.$ Thus, we have
$\alpha_{i_{1}}\in J_{1}$ if and only if the Dynkin diagram of $S\setminus\{\alpha_{i_{1}}\}$ is connected. Consider $c'\,=\,s_{i_{1}}c.$ Then, $c'\,=\,
s_{i_{2}}\cdots s_{i_{n}}$ is a reduced expression. Further, by induction on the rank of $G$
it is deduced that
$\alpha_{i_{r}}\,\in\, J_{r}$ for all $2\,\le \,r\,\le\, n-1$ if and only if
$(c')^{-1}(\sum_{j=2}^{n}\alpha_{i_{j}})\,\in\, -(S\setminus \{\alpha_{i_{1}}\}).$ Hence, we
have $\alpha_{i_{r}}\,\in\, J_{r}$ for all $1\,\le\, r \,\le\, n-1$ if and only if
$c^{-1}(\sum_{t=1}^{n}\alpha_{t})\,\in\, -S$.
\end{proof}

\begin{lemma}\label{lem5.9}
Let $G$ be a simple algebraic group such that rank of $G$ is at least two. Recall that $J_{1}\,=\,\{\alpha\,\in\, S\, \big\vert\, \langle \sum_{t=1}^{n}\alpha_{t},\,
\alpha\rangle\,=\,1\}.$ Then we have the following:
\begin{itemize}
\item[(1)] If $G$ is simply-laced, then the cardinality of $J_{1}$ is at least two. 
	
\item [(2)] If $G$ is not simply-laced, then the cardinality of $J_{1}$ is one.
	
\item [(3)] Fix an integer $1\,\le\, i\,\le \,n.$ If $c$ is a Coxeter element such that $c^{-1}(\sum_{t=1}^{n}\alpha_{t})
\,=\,-\alpha_{i},$ then we have $R^{+}(c^{-1})\cap S\,=\, J_{1}\cap (S\setminus\{\alpha_{i}\}).$ 
\end{itemize}
\end{lemma}

\begin{proof}
Proof of (1) and (2): Follows from the classification of connected Dynkin diagram (see \cite[p. 58]{Hum1}).

Proof of (3): We assume that rank of $G$ is at least two.

Let $\alpha_{j}\in R^{+}(c^{-1})\cap S.$ Note that since $c^{-1}(\sum_{t=1}^{n}\alpha_{t})=-\alpha_{i},$ we have $\ell(cs_{i})=\ell(c)-1,$ and $i_{n}=i.$ Since the Dynkin diagram of $G$ is connected, there is $1\le j\le n-1$ such that $s_{i_{j}}$ does not commute with $s_{i}.$ So, we have $\alpha_{i}\notin R^{+}(c^{-1}).$ Hence, we have $j\neq i.$ Now, if $\langle \sum_{t=1}^{n}\alpha_{t},\alpha_{j}\rangle\le 0,$ then the coefficient of $\alpha_{j}$ in $c^{-1}(\sum_{t=1}^{n}\alpha_{t})=(s_{j}c)^{-1}s_{j}(\sum_{t=1}^{n}\alpha_{t})$ is the coefficient of $\alpha_{j}$ in $s_{j}(\sum_{t=1}^{n}\alpha_{t})$ which is at least one. This is a contradiction as $c^{-1}(\sum_{t=1}^{n}\alpha_{t})=-\alpha_{i}.$ Thus, we have $\langle\sum_{t=1}^{n}\alpha_{t},\alpha_{j}\rangle\ge 1.$ If $\langle \sum_{t=1}^{n}\alpha_{t},\alpha_{j} \rangle\ge 2,$ then $s_{j}(\sum_{t=1}^{n}\alpha_{t})\notin R.$ Therefore, we have $\alpha_{j}\in J_{1}\cap (S\setminus \{\alpha_{i}\}).$

Conversely, let $\alpha_{j}\,\in\, J_{1}\cap (S\setminus\{\alpha_{i}\}).$ Then we have $\langle \sum_{t=1}^{n}\alpha_{t},\,\alpha_{j} 
\rangle\,=\,1.$ Let $c\,=\,s_{i_{1}}\cdots s_{i_{n}}$ be a reduced expression. By using Lemma \ref{lem5.10} we have $\alpha_{i_{1}}\,\in\, 
J_{1}.$ If $i_{1}\,=\,j,$ we are done. So, assume that $j\,\neq\, i_{1}.$ Let $c'\,=\,s_{i_{1}}c.$ Then we have 
$(c')^{-1}(\sum_{k=2}^{n}\alpha_{i_{k}})\,=\,-\alpha_{i}.$ Further, since $\alpha_{j}\,\in\, J_{1}$ and $i_{1}\,\neq \,j$ we have $\langle 
\sum_{k=2}^{n}\alpha_{i_{k}},\,\alpha_{j} \rangle\,\ge\, 1.$ Therefore, by the above argument we have $\alpha_{j}\,\in\, J_{2}\cap 
(S\setminus\{\alpha_{i},\,\alpha_{i_1}\}).$ By induction on the rank of $G,$ we have $$\alpha_{j}\,\in\, J_{2}\cap 
(S\setminus\{\alpha_{i},\,\alpha_{i_{1}}\})\,=\,R^{+}((c')^{-1})\cap (S\setminus\{\alpha_{i_{1}}\}).$$ Further, since $\alpha_{j}\,\in\, J_{1},$ 
it follows that $\langle \alpha_{i_1},\,\alpha_{j}\rangle\,=\,0.$ Therefore, we have $\alpha_{j}\,\in\, R^{+}(c^{-1}).$
\end{proof}

\begin{lemma}\label{lem 5.9}
Let $G$ be a simple algebraic group such that rank of $G$ is at least two. Then the following is
the (non-constructive) list of Coxeter elements $c$ satisfying 
$c^{-1}(\sum_{t=1}^{n}\alpha_{t})\,\in\, -S$:
\begin{itemize}
\item[(1)] If $G$ is simply-laced, for every simple root $\alpha_{i},$ there is a unique Coxeter element
$c_{i}$ such that $c_{i}^{-1}(\sum_{t=1}^{n}\alpha_{t})\,=\,-\alpha_{i}.$
		
\item[(2)] If $G$ is not simply-laced, for every short simple root $\alpha_{i}$, there is a unique Coxeter element $c_{i}$ such
that $c_{i}^{-1}(\sum_{t=1}^{n}\alpha_{t})\,=\,-\alpha_{i}.$ For long simple root $\alpha_{i}$ there is no such Coxeter element.
\end{itemize}
\end{lemma}

\begin{proof}
Uniqueness:\,
Let $c_{i},\,c_{i}'$ be two Coxeter element such that $$c_{i}^{-1}(\sum_{t=1}^{n}\alpha_{t})\,=\,-\alpha_{i}
\,=\,c_{i}'^{-1}(\sum_{t=1}^{n}\alpha_{t}).$$ By Lemma \ref{lem5.9}(3), there is an $i_{1}\,\neq\, i$ such that
$\alpha_{i_{1}}\,\in\, R^{+}(c_{i}^{-1})\cap J_{1}\,=\,R^{+}(c_{i}'^{-1})\cap J_{1}.$ Hence, $s_{i_{1}}c$ and $s_{i_{1}}c_{i}'$ are
two Coxeter elements of the Weyl group $W_{S\setminus\{\alpha_{i_{1}}\}}$ of
smaller rank satisfying $$(s_{i_{1}}c_{i})^{-1}(\sum_{t\neq i_{1}}^{n}\alpha_{t})
\,=\,-\alpha_{i}\,=\, (s_{i_{1}}c_{i}')^{-1}(\sum_{t\neq i_{1}}^{n}\alpha_{t}).$$
Therefore, by induction on the rank it follows that $s_{i_{1}}c_{i}\,=\,s_{i_{1}}c_{i}'.$ Thus, we have $c_{i}\,=\,c_{i}'.$
	
Proof of (1):\, Fix a simple root $\alpha_{i}.$
	By Lemma \ref{lem5.9}(1) it follows that there is $\alpha_{i_{1}}\,\neq\, \alpha_{i}$ such that
$\langle \sum_{t=1}^{n}\alpha_{t},\,\alpha_{i_{1}} \rangle \,=\,1.$ Therefore, the Dynkin diagram of $S\setminus\{\alpha_{i_{1}}\}$ is connected. Now by induction on the rank of $G,$ there is a unique Coxeter element $c_{i}'$ of the Weyl group corresponding to $S\setminus\{\alpha_{i_{1}}\}$ such that $c_{i}'^{-1}(\sum_{t\neq i_1}^{n}\alpha_{t})
\,=\,-\alpha_{i}.$ Therefore, $c_{i}\,=\,s_{i_{1}}c_{i}'$ is a unique Coxeter element such that $c_{i}^{-1}(\sum_{t=1}^{n}\alpha_{t})=-\alpha_{i}.$
	
	Proof of (2):\, We prove by case by case analysis. Since $G$ is
simple, the Dynkin diagram of $G$ is connected. Therefore, $\sum_{t=1}^{n}\alpha_{t}$ is a
root. Further, by the classification of Dynkin diagram, $\sum_{t=1}^{n}\alpha_{t}$ is a
short root whenever $G$ is not simply-laced. Since $\langle \cdot,\,
\cdot \rangle$ is $W$-invariant, for any long simple root $\alpha_{i}$ there is no
Coxeter element $c$ such that $c^{-1}(\sum_{t=1}^{n}\alpha_{t})\,=\,-\alpha_{i}$.

{Case I. $G$ is of type $B_n$ ($n\,\ge\, 2$):}\,
	Note that $\sum_{t=1}^{n}\alpha_{t}$ is the highest short root and $\alpha_{n}$ is the only simple short root for type $B_{n}.$
Then $c_{n}\,=\,s_{1}\cdots s_{n}$  is a unique Coxeter element such that $c_{n}^{-1}(\sum_{t=1}^{n}\alpha_{t})\,=\, -\alpha_{n}.$ 
	
	{Case II. $G$ is of type $C_n$ ($n\ge 3$):}\, 
	Fix an integer $1\le i\le n-1.$ Note that $\alpha_{n}\in J_{1},$ i.e., $\langle \sum_{t=1}^{n}\alpha_{t},\alpha_{n}\rangle=1.$ By (1) there is a unique Coxeter element $c_{i}'\,=\,s_{i_{1}}s_{i_{2}}\cdots s_{i_{n-1}}$ of the Weyl group corresponding to $S\setminus \{\alpha_{n}\}$ such that $c_{i}'^{-1}(\sum_{t=1}^{n-1}\alpha_{t})\,=\,-\alpha_{i}.$ Therefore, we have $c_{i}
\,=\,s_{n}c_{i}'$ is a unique Coxeter element such that $c_{i}^{-1}(\sum_{t\neq i_1}^{n}\alpha_{t})\,=\,-\alpha_{i}.$ 
	
	{Case III. $G$ is of type $F_4$:}\, Note that $\sum_{t=1}^{4}\alpha_{t}$ is a short root for type $F_{4}.$ So, any Coxeter element $c$ satisfying $c^{-1}(\sum_{t=1}^{4}\alpha_{t})\in -S,$ is either $c^{-1}(\sum_{t=1}^{4}\alpha_{t})=-\alpha_{3}$ or $c^{-1}(\sum_{t=1}^{4}\alpha_{t})=-\alpha_{4}.$ Observe that $c_{4}=s_1s_2s_3s_4$ is a unique Coxeter element such that $c_{4}^{-1}(\sum_{t=1}^{4}\alpha_{t})=-\alpha_{4},$ and $c_{3}\,=\,s_1s_2s_4s_3$ is a unique  Coxeter element such that $c_{3}^{-1}(\sum_{t=1}^{4}\alpha_{t})
\,=\,-\alpha_{3}.$
	
{Case IV. $G$ is of type $G_2$ :}\, Note that $\alpha_{1}+\alpha_{2}$ is a short root and  $\alpha_{1}$ is the only short simple root for $G_2.$ Therefore, $c_{1}=s_2s_1$ is a unique  Coxeter element such that $c_{1}^{-1}(\alpha_{1}+\alpha_{2})=-\alpha_{1}.$
\end{proof}

\subsection{Minuscule Weyl group elements}
Recall that a fundamental weight $\omega_{m}$ is said to be minuscule if $\langle \omega_{m}, \beta \rangle\le 1$ for all $\beta \in R^{+}.$
Let $P$ be a minuscule maximal parabolic subgroup of $G,$ i.e., $P\,=\,P_{S\setminus\{\alpha_{m}\}},$ where $\omega_{m}$ is a
minuscule fundamental weight. The elements of $W^{S\setminus \{\alpha_{m}\}}$ are called minuscule Weyl group elements. Now we prove that for any reduced expression $\underline{i}$ of a minuscule Weyl group element $w,$ the anti-canonical line bundle
on $Z(w,\,\underline{i})$ is globally generated. As a consequence, we prove that $Z(w,\,\underline{i})$ is weak Fano. 

Take any $w\,\in\, W^{S\setminus\{\alpha_{m}\}}$ such that $w\,\neq\, id.$ Without loss of generality we may assume that ${\rm supp}(w)
\,=\,S.$ Indeed, if 
${\rm supp}(w)\,\neq\, S,$ consider the subgroup $G_{w}$ of $G$ generated by the $U_{\pm \alpha},$ where $\alpha \,\in\, {\rm supp}(w).$ This is the derived 
subgroup of the Levi subgroup $L_{{\rm supp}(w)},$ and hence it is a semi-simple subgroup of $G,$ normalized by $T$ and contains a 
representative of $w.$ Consider the associated Schubert variety $X(w)\,\subset \,G/P.$ Note that $P\bigcap G_{w}$ is a minuscule maximal parabolic 
subgroup of $G_{w},$ and we have
$$X(w)\,=\,\overline{BwP}/P\,\simeq\, \overline{(B\cap G_{w})w(P\cap G_{w})}/(P \cap G_{w})\,\subset\, G_{w} /(P\cap 
G_{w}).$$

Let $w\,=\,s_{i_{1}}s_{i_{2}}\cdots s_{i_{r}}$ be a reduced expression and $\underline{i}\,=\,(i_{1},i_{2},\, \cdots ,\,i_{r}).$  Then

\begin{theorem}\label{lem1.3}
	The anti-canonical line bundle $K_{Z(w,\underline{i})}^{-1}$ is globally generated.
	In particular, we have
	$$H^j(Z(w,\,\underline{i}),\, K_{Z(w,\underline{i})}^{-1})\,=\,0$$ for all $j\,\ge\, 1.$
\end{theorem}
\begin{proof}
	By descending induction on $1\le j\le r,$ we prove that $m_{j}\ge 0.$ By
using Proposition \ref{lem5.3}, we have $m_r\,=\,2.$ We fix $1\,\le\, j_{0}\,\le\, r-1.$ We
assume that $m_{j}\,\ge\, 0$ for all $j_{0}+1\,\le\, j\,\le\, r.$
	
	{Case 1:} There is no $j_{0}+1\,\le\, k\,\le\, r,$ such that $i_{k}\,=\,i_{j_{0}}.$ Let $m\,=\,i_{r}.$ Let $v\,=\,s_{i_{j_{0}+1}}\cdots
s_{i_{r}}.$ Now we have $\langle v(\omega_{m}),
\,\alpha_{i_{0}}\rangle\,=\,1,$ as $\omega_{m}$ is minuscule. Further, since $\omega_{m}$
is minuscule, we have $v(\omega_{m})\,=\,\omega_{m}-\sum_{k=j_{0}+1}^{r}\alpha_{i_{k}}.$
Therefore, since $\langle \sum_{k=j_{0}+1}^{r}\alpha_{i_{k}},\, \alpha_{i_{j_{0}}}\rangle
\,=\,-1,$ by Proposition \ref{lem5.3} we have $m_{j_{0}}\,=\,\langle
\sum_{k=j_{0}}^{r}\alpha_{i_{k}},\, \alpha_{i_{j_{0}}}\rangle\,=\,2-1\,=\,1.$
	
	{ Case 2:} There is an integer $j_{0}+1\le k\le r$ such that $i_{k}=i_{j_{0}}.$ Let $k$ be the least such integer. Let $v_{1}=s_{j_{0}+1}\cdots s_{i_{r}}$ and $v_{2}=s_{i_{k}}\cdots s_{i_{r}}.$ Then $\langle v_{2}(\omega_{m}),\alpha_{i_{j_{0}}}\rangle=\langle v_{2}(\omega_{m}),\alpha_{i_{k}}\rangle=-1.$ On the other hand, $\langle v_{1}(\omega_{m}), \alpha_{i_{j_{0}}}\rangle=1.$ Therefore, by Proposition \ref{lem5.3} we have $m_{j_{0}}=\langle \sum_{l=j_{0}}^{k-1}\alpha_{i_{l}},\alpha_{i_{j_{0}}} \rangle=2+ \langle \sum_{l=j_{0}+1}^{k-1}\alpha_{i_{l}},\alpha_{i_{j_{0}}} \rangle=2+ \langle v_{2}(\omega_{m}),\alpha_{i_{j_{0}}}\rangle-\langle v_{1}(\omega_{m}),\alpha_{i_{j_{0}}}\rangle=2-1-1=0.$ By \cite[p. 301, Theorem 6]{LLM} vanishing results follow immediately 
\end{proof}

\begin{corollary} \label{prop5.5}
	Let $w$ and $\underline{i}$ be as above. Then $Z(w,\,\underline{i})$ is weak-Fano. 
\end{corollary}

\begin{proof}
By Corollary \ref{cor2.3}, the line bundle $K_{Z(w,\underline{i})}^{-1}$ is big. On the other
hand, by Theorem \ref{lem1.3}, the line bundle $K_{Z(w,\underline{i})}^{-1}$ is globally
generated. In particular, 
$K_{Z(w,\underline{i})}^{-1}$ is nef (see \cite[p. 52, Example 1.4.5.]{L}). Therefore, $Z(w,
\underline{i})$ is weak-Fano.
\end{proof}

\section{Some examples in type $A_{n-1}$}

In this section we give some examples of the reduced expressions of the longest element $w_0$ for type $A_{n-1}$ for which 
anti-canonical line bundle is globally generated.
	
\begin{example}\label{prop6.1}
	Let $G$ be of type $A_{n-1}.$ Consider the reduced expression $$w_{0}\,=\,s_{1}(s_{2}s_{1})(s_{3}s_{2}s_{1})\cdots (s_{n-1}s_{n-2}\cdots s_{1}).$$ Let $\underline{i}\,=\,(i_{1},\,\cdots ,\,i_{N})$ be the tuple associated to the above reduced expression of
	$w_{0},$ where $N\,=\,{{n}\choose{2}}.$ Observe that for every integer
$1\,\le\, i\,\le \,n-1,$ the number of distinct $s_{j}$'s that are not commuting with $s_{i}$ and
appearing between two consecutive appearances of $s_{i}$ in the above reduced expression of $w_{0}$ is at most two.
	Hence, by using Proposition \ref{lem5.3} we have $m_{j}\,\ge\, 0$ for all $1
\,\le\, j\,\le\, N.$ Therefore, by \cite[p. 465, Corollary 3.3]{LT} it follows that
the anti-canonical line bundle $K_{Z(w_0,\underline{i})}^{-1}$ on $Z(w_{0},\,\underline{i})$ is globally generated. 
	\end{example}

	\begin{example}\label{cor6.2}
	Let $G$ be of type $A_{n-1}.$	Consider the reduced expression $$w_{0}\,=\,s_{n-1}(s_{n-2}s_{n-1})\cdots (s_{1}\cdots s_{n-1}).$$ Let
		$\underline{i}$ be the tuple associated to it. Then the anti-canonical line bundle on $Z(w_{0},\,\underline{i})$ is globally generated. 
	\end{example}

	\begin{example}
		Let $G$ be of type $A_{2}.$ Then  $Z(w_{0},\,\underline{i})$ is Fano for any reduced expression $\underline{i}$ of $w_{0}$.
	\end{example}

	\begin{proposition}\label{prop6.4}
		Assume that $G\,=\,{\rm SL}(4,\mathbb{C}).$ Let $\underline{i}$ be the tuple associated to a reduced expression of $w_{0}.$
		Then the anti-canonical line bundle on $Z(w_{0},\,\underline{i})$ is globally generated. 
	\end{proposition}
	
	\begin{proof}
		By \cite[p.~439, Theorem 3.1]{CK} the BSDH-varieties associated to two reduced expressions of $w$ differing only by commuting
		relations are isomorphic. Therefore, it is enough to consider the expressions up to commuting relations.
		
		Now, we argue case by case. In each case we prove that $m_{j}$'s are non-negative for one sub-case and the proof for the other sub-cases are similar.
		
		Case I:\, First observe that up to commuting relations there are exactly three reduced expressions of $w_{0}$ ending with $s_{1},$ namely
		\begin{itemize}
			\item $s_{1}s_{2}s_{1}s_{3}s_{2}s_{1}$,
			
			\item $s_{2}s_{1}s_{2}s_{3}s_{2}s_{1}$ and
			
			\item $s_{2}s_{3}s_{1}s_{2}s_{3}s_{1}.$
		\end{itemize} 
		
		Consider the reduced expression $w_{0}\,=\,s_{1}s_{2}s_{1}s_{3}s_{2}s_{1}$ and $\underline{i}\,=\,(1,\,2,\,1,\,3,\,2,\,1).$
		Using Proposition \ref{lem5.3} we see that $(m_{1},\,m_{2},\,m_{3},\,m_{4},\,m_{5},\,m_{6})=\,(1,\,0,\,1,\,1,\,1,\,2).$

		Case II:\, Observe that up to commuting relations there are exactly three reduced expressions of $w_{0}$ ending with $s_{2},$ namely
		\begin{itemize}
			\item $s_{1}s_{2}s_{3}s_{2}s_{1}s_{2}$,
			
			\item $s_{3}s_{2}s_{1}s_{2}s_{3}s_{2}$ and
			
			\item $s_{1}s_{3}s_{2}s_{3}s_{1}s_{2}.$
		\end{itemize} 
		
		Consider the reduced expression $w_{0}\,=\,s_{1}s_{2}s_{3}s_{2}s_{1}s_{2}$ and $\underline{i}\,=\,(1,\,2,\,3,\,2,\,1,\,2).$ Then by using Proposition \ref{lem5.3} we see that $(m_{1},\,m_{2},\,m_{3},\,m_{4},\,m_{5},\,m_{6})\,=\,(0,\,1,\,0,\,1,\,1,\,2).$

		Case III:\, Observe that up to commuting relations there are exactly three reduced expressions of $w_{0}$ ending with $s_{3},$ namely
		\begin{itemize}
			\item $s_{2}s_{3}s_{1}s_{2}s_{1}s_{3}$,
			
			\item $s_{3}s_{2}s_{3}s_{1}s_{2}s_{3}$ and
			
			\item $s_{2}s_{3}s_{2}s_{1}s_{2}s_{3}$ 
			
		\end{itemize} 
		
		Consider the reduced expression $w_{0}\,=\,s_{2}s_{3}s_{1}s_{2}s_{1}s_{3}$ and $\underline{i}\,=\,(2,\,3,\,1,\,2,\,1,\,3).$ Then by using Proposition \ref{lem5.3} we see that $(m_{1},\,m_{2},\,m_{3},\,m_{4},\,m_{5},\,m_{6})\,=\,(0,\,1,\,1,\,0,\,2,\,2).$

Therefore, by using \cite[p. 465, Corollary 3.3]{LT} and combining the above three cases proof of the proposition follows.
	\end{proof}
	
	We give some examples showing that Proposition \ref{prop6.4} is not
	valid for $G\,=\,{\rm SL}(n,\mathbb{C})$ when $n\,\ge\, 5$.
	
	\begin{example}\label{ex6.5}
		Set $G\,=\,{\rm SL}(5,\mathbb{C}).$
			 Consider $w_0\,=\,s_{3}s_{2}s_{1}s_{4}s_{3}s_{2}s_{3}s_{1}s_{4}s_{3}.$ 
			 
			 Let		 
			  $\underline{i}\,=(3,\, 2,\, 1,\, 4,\, 3,\, 2,\, 3,\, 1,\, 4,\, 3).$ Then
the anti-canonical line bundle 
$K_{Z(w_0,\underline{i})}^{-1}$ is not globally generated because by using Proposition \ref{lem5.3} we see that
$$(m_{1},\,m_{2},\,m_{3},\,m_{4},\,m_{5},\,m_{6},\,m_{7},\,m_{8},\,m_{9},\,m_{10})\, =\,(\,0,\,0,\,1,\,0,\,1,\,-1,\,1,\,2,\,1,\,2).$$ Therefore, by \cite[p. 465, Corollary 3.3]{LT} $K_{Z(w_{0},\underline{i})}^{-1}$ is not globally generated. But, 
by \cite[p.~472, Theorem 7.4]{LT} (or by using Lemma \ref{lem3.1}(ii)), all the higher
cohomology groups vanish.
	\end{example}
	
	\begin{example}
		Set $G\,=\,{\rm SL}(6,\mathbb{C}).$ Consider $w\,=\,s_{3}s_{4}s_{5}s_{4}s_{2}s_{1}s_{2}$ and $$\underline{i}\,=\,(3,\,4,\,5,\,4,\,2,\,1,\,2).$$
		Then the anti-canonical line bundle $K_{Z(w,\underline{i})}^{-1}$ is not globally generated because by using Proposition \ref{lem5.3} we see that $(m_{1},\,m_{2},\,m_{3},\,m_{4},\,m_{5},\,m_{6},\,m_{7})\,=\,(-2,\,1,\,1,\,2,\,1,\,1,\,2).$ Further,
		using Lemma \ref{lem3.1}, we see that $H^1(Z(w,\,\underline{i}), \,K_{Z(w,\underline{i})}^{-1})\,\neq\, 0.$
	\end{example}
	
	\begin{remark}\label{rm6.7}
		The anti-canonical line bundle of BSDH variety $Z(w,\,\underline{i})$ not only depends on $w$ but also
		it depends on the chosen reduced expression $\underline{i}.$ For example, this can be seen from
		Example \ref{prop6.1} and Example \ref{ex6.5}. 
	\end{remark}
	\section*{Acknowledgements}
	
We are very grateful to the 
referees for careful reading, many valuable comments and suggestions to improve the 
exposition of the article. We thank a referee for suggesting Proposition \ref{prop4.0} and 
providing its proof.
The second named author would like to thank the Infosys Foundation for the partial financial 
support. The first named author is partially supported by a J. C. Bose Fellowship.

\end{document}